\documentclass[11pt,reqno,oneside]{amsart}

\usepackage[utf8]{inputenc}
\usepackage[T1]{fontenc}
\usepackage[top=3cm,bottom=3cm,left=3cm,right=3cm]{geometry}

\usepackage{hyperref}
\usepackage{mathtools}
\usepackage{thmtools}
\usepackage[capitalize, nameinlink]{cleveref}

\usepackage{xcolor}
\definecolor{red}{rgb}{1,0,0}
\hypersetup{
    colorlinks,
    linkcolor={red!50!black},
    citecolor={blue!50!black},
    urlcolor={blue!80!black}
}

\usepackage{amsfonts}
\usepackage{amssymb}
\usepackage{latexsym}
\usepackage{mathrsfs}
\usepackage{textcomp}
\usepackage{longtable}
\usepackage{booktabs}
\usepackage{graphicx}
\usepackage{setspace}
\usepackage{nicefrac}
\usepackage{enumerate}
\usepackage{wasysym}
\usepackage[normalem]{ulem}
\usepackage{upgreek}
\usepackage{paralist}
\usepackage{etoolbox}
\usepackage{mathdots}
\usepackage{wrapfig}
\usepackage{floatflt}
\usepackage{tensor} 
\usepackage{lscape}
\usepackage{stmaryrd}

\usepackage{tikz}
\usetikzlibrary{arrows}
\usetikzlibrary{matrix}
\usetikzlibrary{positioning}
\usetikzlibrary{calc}

\usepackage[all,cmtip]{xy} 
\usepackage[labelfont=bf, font={small}]{caption}[2005/07/16]

\newcommand{\vvirg}{ , \dots , }
\newcommand{\ootimes}{ \otimes \cdots \otimes }

\newcommand{\ttimes}{ \times \cdots \times }

\newcommand{\textfrac}[2]{{\textstyle \frac{#1}{#2}}}

\newcommand{\bfA}{\mathbf{A}}

\newcommand{\bfD}{\mathbf{D}}

\newcommand{\bfF}{\mathbf{F}}

\newcommand{\bfT}{\mathbf{T}}

\newcommand{\bfX}{\mathbf{X}}

\newcommand{\bfj}{\mathbf{j}}

\newcommand{\bfn}{\mathbf{n}}

\newcommand{\bfp}{\mathbf{p}}

\newcommand{\bfu}{\mathbf{u}}

\newcommand{\calA}{\mathcal{A}}
\newcommand{\calB}{\mathcal{B}}
\newcommand{\calC}{\mathcal{C}}
\newcommand{\calD}{\mathcal{D}}

\newcommand{\calF}{\mathcal{F}}
\newcommand{\calG}{\mathcal{G}}

\newcommand{\calI}{\mathcal{I}}

\newcommand{\calL}{\mathcal{L}}
\newcommand{\calM}{\mathcal{M}}
\newcommand{\calN}{\mathcal{N}}

\newcommand{\calQ}{\mathcal{Q}}
\newcommand{\calR}{\mathcal{R}}
\newcommand{\calS}{\mathcal{S}}
\newcommand{\calT}{\mathcal{T}}
\newcommand{\calU}{\mathcal{U}}

\newcommand{\calY}{\mathcal{Y}}
\newcommand{\calZ}{\mathcal{Z}}

\newcommand{\bbC}{\mathbb{C}}

\newcommand{\bbN}{\mathbb{N}}

\newcommand{\bbP}{\mathbb{P}}

\newcommand{\bbZ}{\mathbb{Z}}

\newcommand{\frakb}{\mathfrak{b}}

\newcommand{\frakg}{\mathfrak{g}}

\newcommand{\frakm}{\mathfrak{m}}

\newcommand{\fraks}{\mathfrak{s}}

\newcommand{\frakz}{\mathfrak{z}}

\newcommand{\rmQ}{\mathrm{Q}}
\newcommand{\rmR}{\mathrm{R}}

\newcommand{\rmd}{\mathrm{d}}

\renewcommand{\phi}{\varphi}
\newcommand{\eps}{\varepsilon}
\renewcommand{\theta}{\vartheta}

\renewcommand{\tilde}[1]{\widetilde{#1}}

\renewcommand{\bar}[1]{\overline{#1}}

\newcommand{\id}{\mathrm{id}}

\newcommand{\rank}{\mathrm{rank}}
  
\newcommand{\image}{\mathrm{image}}  
\DeclareMathOperator{\codim}{codim}

\DeclareMathOperator{\Hom}{Hom}

\DeclareMathOperator{\End}{End}

\newcommand{\SL}{\mathrm{SL}}

\newcommand{\Gr}{\mathrm{Gr}}

\newcommand{\GL}{\mathrm{GL}}

\newcommand{\supp}{\mathrm{supp}}

\newcommand{\uR}{\underline{\rmR}}
\newcommand{\uQ}{\underline{\rmQ}}

\newcommand{\GR}{\mathrm{GR}}

\newcommand{\MaMu}{\mathsf{MaMu}}

\newcommand{\fraksl}{\mathfrak{sl}}

\newcommand{\Mat}{\mathrm{Mat}}

\DeclareMathOperator{\Stab}{Stab}

\newcommand{\fillwidthof}[3][c]
	{%
		\parbox
		{%
			\widthof{#2}%
		}%
		{%
			\ifx#1c%
				\centering#3%
			\else\ifx#1l%
				#3\hfill%
			\else\ifx#1r%
				\hfill#3%
			\fi\fi\fi%
		}%
	}%

\def\mylettrine#1#2 {\lettrine{#1}{#2}\space}

\newcommand{\partinto}[1][]{\smash{\mathord{\mathchoice{%
  \xymatrix@=0.4em@1{%
  \ar@{|-}[rr]_-*--{\scriptstyle #1}
  &*{\phantom{\scriptstyle{#1}}}&}
}{
  \xymatrix@=0.25em@1{%
  \ar@{|-}[rr]_-*--{\scriptstyle #1}
  &*{\phantom{\scriptstyle{#1}}}&}
}{
  \xymatrix@=0.2em@1{%
  \ar@{|-}[rr]_-*--{\scriptscriptstyle #1}
  &*{\phantom{\scriptscriptstyle{#1}}}&}
}{}}}}
\newcommand{\partintonosmash}[1][]{\mathord{\mathchoice{%
  \xymatrix@=0.4em@1{%
  \ar@{|-}[rr]_-*--{\scriptstyle #1}
  &*{\phantom{\scriptstyle{#1}}}&}
}{
  \xymatrix@=0.25em@1{%
  \ar@{|-}[rr]_-*--{\scriptstyle #1}
  &*{\phantom{\scriptstyle{#1}}}&}
}{
  \xymatrix@=0.2em@1{%
  \ar@{|-}[rr]_-*--{\scriptscriptstyle #1}
  &*{\phantom{\scriptscriptstyle{#1}}}&}
}{}}}
\newcommand{\partintostar}[1][]{\smash{\mathord{\mathchoice{%
  \xymatrix@=0.4em@1{%
  \ar@{|-}[rr]_-*--{\scriptstyle #1}^-*--{\scriptstyle \ast}
  &*{\phantom{\scriptstyle{#1}}}&}
}{
  \xymatrix@=0.25em@1{%
  \ar@{|-}[rr]_-*--{\scriptstyle #1}^-*--{\scriptstyle \ast}
  &*{\phantom{\scriptstyle{#1}}}&}
}{
  \xymatrix@=0.2em@1{%
  \ar@{|-}[rr]_-*--{\scriptscriptstyle #1}^-*--{\scriptstyle \ast}
  &*{\phantom{\scriptscriptstyle{#1}}}&}
}{}}}}
\newcommand{\partintostarnosmash}[1][]{\mathord{\mathchoice{%
  \xymatrix@=0.4em@1{%
  \ar@{|-}[rr]_-*--{\scriptstyle #1}^-*--{\scriptstyle \ast}
  &*{\phantom{\scriptstyle{#1}}}&}
}{
  \xymatrix@=0.25em@1{%
  \ar@{|-}[rr]_-*--{\scriptstyle #1}^-*--{\scriptstyle \ast}
  &*{\phantom{\scriptstyle{#1}}}&}
}{
  \xymatrix@=0.2em@1{%
  \ar@{|-}[rr]_-*--{\scriptscriptstyle #1}^-*--{\scriptstyle \ast}
  &*{\phantom{\scriptscriptstyle{#1}}}&}
}{}}}

\newtheorem{theorem}{Theorem}[section]

\newtheorem{lemma}[theorem]{Lemma}
\newtheorem{proposition}[theorem]{Proposition}
\newtheorem{corollary}[theorem]{Corollary}
\theoremstyle{definition}

\newtheorem{remark}[theorem]{Remark}
\newtheorem{example}[theorem]{Example}

\crefname{theorem}{Theorem}{Theorems}
\crefname{lemma}{Lemma}{Lemmas}
\crefname{proposition}{Proposition}{Propositions}
\crefname{corollary}{Corollary}{Corollaries}
\crefname{definition}{Definition}{Definitions}
\crefname{remark}{Remark}{Remarks}
\crefname{example}{Example}{Examples}

\numberwithin{equation}{section}

\newcommand{\Flag}{\calF \mathit{lag}}

\title{Border subrank of higher order tensors and algebras}

\author{Chia-Yu Chang}
\author{Fulvio Gesmundo}
\address[C.-Y. Chang, F. Gesmundo]{Institut de Mathématiques de Toulouse; UMR5219 -- Université de Toulouse; CNRS -- UPS, F-31062 Toulouse Cedex 9, France}
\email{chia-yu.chang@math.univ-toulouse.fr}
\email{fgesmund@math.univ-toulouse.fr}

\author{Jeroen Zuiddam}
\address[J. Zuiddam]{University of Amsterdam}
\email{j.zuiddam@uva.nl}

\subjclass[2020]{15A69, 68Q17, 13A50, 16S80}
\keywords{subrank, border subrank, structure tensor, geometric rank, associative algebra}

\begin{document}
\begin{abstract}
We determine the border subrank of higher order structure tensors of several families of algebras, and in particular obtain the following results. (1) We determine tight bounds on the border subrank of $k$-fold matrix multiplication and $k$-fold upper triangular matrix multiplication for all $k$. (2) We determine the border subrank of the higher order structure tensors of truncated polynomial algebras, null algebras, and apolar algebras of a quadric. (3) We determine the border subrank of the higher order structure tensors of the Lie algebra $\mathfrak{sl}_2$ for all orders. (4) We prove that degeneration of structure tensors of algebras propagates from higher to lower order. Along the way, we investigate which upper bound methods (geometric rank, $G$-stable rank, socle degree) are effective in which settings, and how they relate. Our work extends the results of Strassen (J.~Reine Angew.~Math., 1987, 1991), who determined the asymptotic subrank of these algebras for tensors of order three, in two directions: we determine the border subrank itself rather than its asymptotic version, and we consider higher order structure tensors.
\end{abstract}

\maketitle

\setcounter{tocdepth}{1}
\tableofcontents

\section{Introduction}\label{sec: intro}

\emph{Subrank} and \emph{border subrank} are fundamental invariants of tensors, introduced by Strassen \cite{Strassen87} in the context of matrix multiplication complexity. They measure how well a given tensor can be transformed into a diagonal tensor via multilinear maps, and are central to Strassen's theory of asymptotic spectra, with applications across theoretical computer science, combinatorics, and quantum information theory; see \cite{MR1440179,blaser2013fast,MR3729273,WigZui:AsymptoticSpectra} for background and surveys.

\subsection*{Subrank and related invariants}
Given a tensor $T \in \bbC^n \ootimes \bbC^n$ of order $k$, the subrank $\rmQ(T)$ is the largest $q\in\bbN$ for which there exist linear maps $A_i : \bbC^n \to \bbC^q$ such that 
\[(A_1 \ootimes A_k)(T) = \sum_{i=1}^q e_i^{\otimes k}.
\]
This generalizes the matrix rank, which can be characterized as the largest $q$ such that a matrix can be transformed into a $q \times q$ identity matrix via row and column operations. The \emph{border subrank} $\uQ(T)$ is the largest $q\in\bbN$ for which the unit tensor $\sum_{i=1}^q e_i^{\otimes k}$ lies in the orbit closure of $T$ under the action of $\End(\bbC^n)^{\times k}$ on $\bbC^n \ootimes \bbC^n$ (\cref{subsec: subrank}). The \emph{asymptotic subrank} $\lim_{N \to \infty} \rmQ(T^{\otimes N})^{\smash{1/N}}$ measures the rate at which the subrank of large tensor powers grows. Symmetric versions of these notions are studied in \cite{ChrFawTaZui,BDNS26}. Related invariants include the geometric rank, the Schmidt rank \cite{KopMosZui:GeomRankSubrankMaMu,chen2025boundsgeometricrankterms,KazhdanLampertPolishchuk:SchmidtRank,lindberg2023symmetricgeometricranksymmetric}, and the $G$-stable rank \cite{Derk:GStableRank}, all of which provide upper bounds on the border subrank. Other notions of rank for higher order tensors include the slice rank \cite{SawTao:NoteSliceRank}, the partition rank \cite{MR4078997,lampert2025slicerankpartitionrank}, and the analytic rank \cite{MR3964143}. 

\subsection*{Higher order structure tensors of algebras}
The \emph{structure tensor} of a finite-dimensional algebra $\calA$ encodes its multiplication map; more generally, the $k$-fold product in $\calA$ defines a multilinear map $\calA^{\times k} \to \calA$, whose associated tensor $\smash{T^{(k)}_\calA}$ of order $k+1$ we call the \emph{$k$-th structure tensor} of $\calA$. In two influential papers, Strassen \cite{Strassen87,Strassen91} determined the \emph{asymptotic} subrank of the structure tensors, in the case $k=2$, of several families of algebras: matrix multiplication, upper triangular matrix multiplication, truncated polynomial algebras $\bbC[x]/(x^d)$, and null algebras $\bbC[x_1 \vvirg x_n]/(x_1 \vvirg x_n)^2$. These results, and in particular the asymptotic subrank of the matrix multiplication tensor, are closely related to the exponent $\omega$ of matrix multiplication and to Strassen's laser method, and to barriers for it \cite{Alman:LimitsUniversalMethod,ChrVraZui:BarriersMaMu}. Strassen's lower bound on the border subrank of matrix multiplication was later sharpened to a precise value by Kopparty--Moshkovitz--Zuiddam \cite{KopMosZui:GeomRankSubrankMaMu}, using \emph{geometric rank}, and this line of work has since been developed further in several directions \cite{chen2025boundsgeometricrankterms,lindberg2023symmetricgeometricranksymmetric,bik2024uniformitylimitstensors}.

Recently, Bl\"aser--Mayer--Shringi \cite{BlaMayShr23} extended the \emph{rank} side of this study to higher order tensors. They gave a systematic treatment of the rank of $\smash{T^{(k)}_\calA}$ as $k$ grows, generalizing the Alder--Strassen bound \cite{AldStr}, and proving a dichotomy result: for a fixed algebra $\calA$, the rank of $\smash{T^{(k)}_\calA}$ grows polynomially as a function of $k$ if the \emph{reduced} algebra $\calA/\mathrm{rad}(\calA)$ is commutative, and grows exponentially otherwise.

In this work, we take up the analogous question on the subrank side: we determine the border subrank of the higher order structure tensors $\smash{T^{(k)}_\calA}$ for several families of algebras, including those studied by Strassen. Unlike rank, the subrank of $\smash{T^{(k)}_\calA}$ is non-increasing as a function of $k$. Hence, it has a stable value, which we also determine for several of these families; in some cases it collapses to $1$. Along the way, we investigate which upper bound methods (geometric rank, $G$-stable rank, socle degree) are effective in which settings, and how they relate, and we prove a general propagation result for degenerations of higher order structure tensors. Our main contributions are listed in \cref{subsec: contributions}.

\subsection*{Further motivation}
The focus on structure tensors of algebras is motivated by several connections.

\emph{Algebra and complexity.} Invariants of the structure tensors $\smash{T^{(k)}_\calA}$ are also invariants of the algebra $\calA$ itself: it was observed in \cite{Poon08,BlaLys} that at order three, degeneration of structure tensors is equivalent to degeneration of algebras in the sense of deformation theory, so algebraic features of $\calA$ serve as proxies for tensor invariants and vice versa. A classical instance of this correspondence is the Alder--Strassen bound \cite{AldStr}, relating the rank of a structure tensor to the number of maximal ideals of the algebra; the bound of \cite{BlaMayShr23} mentioned above is a higher order generalization. Further applications of higher order multilinear complexity include polynomial identity testing \cite{Bsh14} and circuit complexity of higher order tensors \cite{DBLP:journals/corr/abs-2602-11975}.

\emph{Quantum information.} Tensors correspond to pure quantum states, and subrank and border subrank are entanglement measures under stochastic local operations and classical communication (SLOCC) \cite{VrChr}. Iterated matrix multiplication tensors are graph tensors, studied in the context of tensor networks: they model matrix product states \cite{ChrVraZui:AsyRankGraph,BerDLaGes}, and their border subrank has applications in quantum communication complexity \cite{DBLP:conf/innovations/BuhrmanCZ17}.

\emph{Combinatorics.} The structure tensor of $\Bbbk[t]/(t^3)$ encodes the capset problem \cite{MR3583357,MR3583358,Tao:SymmetricFormulationCapset,SawTao:NoteSliceRank}, and higher order analogs $\Bbbk[t]/(t^k)$ arise in the sunflower problem \cite{MR3668469} and multicolored sum-free problems \cite{MR3957831}. Further related work on subrank and border subrank includes the generic case \cite{derksen2022subrankoptimalreductionscalar,pielasa2024exactvaluesgenericsubrank,BCDR25} and the real case \cite{biaggi2025realsubrankorderthreetensors}.

\subsection{Our contributions}\label{subsec: contributions}
We now describe our main contributions in more detail.
\begin{enumerate}
\item \textbf{Matrix multiplication.} We determine tight bounds on the border subrank of $k$-fold (iterated) $n \times n$ matrix multiplication: the border subrank is between $\smash{\frac{1}{k-1}} n^2$ and $\smash{\frac{k+1}{2k}} n^2 + O(k)$ (\cref{thm: main MaMu} and \cref{corol: GR MaMu}). Previously, only the case $k=2$ was known \cite{Strassen87,KopMosZui:GeomRankSubrankMaMu}.

\item \textbf{Upper triangular matrix multiplication.} We determine the border subrank of $k$-fold (iterated) upper triangular matrix multiplication: it is $\frac{1}{2k} n^2 + \Theta(n)$ (\cref{thm: subrank triangular}). Previously, only the asymptotic subrank for the case $k=2$ was known \cite{Strassen87}. In both~(1) and (2), the upper bounds are obtained by determining the geometric rank exactly, and the lower bounds by constructing explicit degenerations.

\item \textbf{Commutative algebras.} We determine the border subrank of the higher order structure tensors of the families of commutative algebras studied by Strassen \cite{Strassen91}, namely truncated polynomial algebras $\bbC[x]/(x^d)$ and null algebras $\bbC[x_1 \vvirg x_n]/(x_1 \vvirg x_n)^2$, as well as apolar algebras of a quadric $\bbC[x_1 \vvirg x_n]/(x_ix_j : i \neq j, \; x_i^2 - x_1^2 : i \neq 1)$, which generalizes the big Coppersmith--Winograd tensors \cite{CopWin} (\cref{prop: null algebra}, \cref{prop: lower bound QTRd}, \cref{thm: subrank higher CW}). The upper bounds use geometric rank, $G$-stable rank, and an algebraic invariant called the socle degree; the lower bounds are obtained by constructing explicit degenerations.

\item \textbf{Lie algebras.} We initiate the study of higher order structure tensors of Lie algebras. We characterize the geometric rank of the Lie bracket in any reductive Lie algebra, and we determine the border subrank of the higher order structure tensors of $\fraksl_2$ (traceless $2 \times 2$ matrices) precisely for all orders (\cref{thm: sl lie algebra}).

\item \textbf{Propagation of degeneration.} We prove that degeneration of higher order structure tensors of algebras propagates to lower order: if the structure tensors of the $k_0$-fold product of two algebras are related by a degeneration, then so are their structure tensors of the $k$-fold product for all $k \leq k_0$ (\cref{thm: degeneration between higher order algebras}). Moreover, if the two algebras have the same dimension, then degeneration at any single $k_0 \geq 2$ implies degeneration at every $k \in \bbN$. This is a consequence of a general geometric characterization (\cref{thm: degenerations via grassmannians}), which generalizes \cite[Thm. 4.3]{CGZGap}.
\end{enumerate}

\subsection*{Acknowledgements} Part of this work was developed while the authors were visiting the Simons Institute for the Theory of Computing, during the program on \emph{Complexity and Linear Algebra}, in Fall 2025. We thank the Simons Institute for its support and for providing a stimulating research environment. We thank Junho Choe for suggestions regarding equations on osculating varieties, Giorgio Ottaviani for helpful discussions on invariants of tensors, and Jarek Buczy\'nski, Jakub Jagie{\l\l}a and Joachim Jelisiejew for our exchanges on the material of \cref{section: flattenings}. JZ was supported by ERC Starting Grant 101220349 and NWO VIDI Grant VI.Vidi.243.195.

\section{Preliminaries and propagation of tensor degeneration}\label{sec: prelim}
In this section, we recall basics on tensors, in particular in the context of finite-dimensional algebras. We prove a generalization of \cite[Thm. 4.3]{CGZGap}, which will allow us to obtain propagation results on degeneration from high order tensors to low order tensors. Throughout the paper $V_1 \vvirg V_k$ are complex vector spaces with $n_i = \dim V_i$. Given a vector space $V$, an (affine) algebraic variety $X \subseteq V$ is a subset of $V$ characterized by the vanishing of a system of polynomial equations on $V$. We assume some basic knowledge of algebraic geometry.

\subsection{Restrictions and degenerations of tensors} The group $\GL(V_1) \times \cdots \times \GL(V_k)$ acts on $V_1 \ootimes V_k$ via
\[
 (g_1 \vvirg g_k) \cdot (v_1 \ootimes v_k) = (g_1 \ootimes g_k)(v_1 \ootimes v_k) = (g_1v_1 \ootimes g_kv_k),
\]
extended linearly. This action extends to a monoidal action of $\End(V_1) \times \cdots \times \End(V_k)$. Let $T,T' \in V_1 \ootimes V_k$ be two tensors. 
\begin{itemize}
\item We say that $T$ and $T'$ are isomorphic if they belong to the same orbit for the action of $\GL(V_1) \ttimes \GL(V_k)$. In other words, there exist invertible linear maps $g_i \in \GL(V_i)$ such that $T' = (g_1 \ootimes g_k) T$. 
\item We say that $T'$ is a restriction of $T$ if $T'$ belongs to the monoidal orbit of $T$ for the action of $\End(V_1) \ttimes \End(V_k)$. In other words, there exist linear maps $X_i \in \End(V_i)$ such that $T' = (X_1 \ootimes X_k) T$.
\item We say that $T'$ is a degeneration of $T$ if $T'$ belongs to the closure of the orbit of $T$ under the action of $\GL(V_1) \ttimes \GL(V_k)$, or equivalently under the action of $\End(V_1) \ttimes \End(V_k)$. By \cite[Thm. 2.33]{MumBook}, the closure can be taken equivalently in the Zariski or the Euclidean topology: in particular, $T$ degenerates to $T'$ if there exist curves $g_i(\eps) \in \GL(V_i)$ such that $T' = \lim_{\eps \to 0} (g_1(\eps) \ootimes g_k(\eps))T$.
\end{itemize}
We use the same terminology if the tensors $T,T'$ belong to different tensor spaces; in that case isomorphism, restriction and degeneration should be understood after re-embedding the tensor spaces into a common large enough tensor product. 

More generally, one can define degeneration in terms of the action of any algebraic group $G$ on a representation $W$: given $w_1, w_2 \in W$, we say that $w_2$ is a $G$-degeneration of $w_1$ if $w_2 \in \bar{ G \cdot w_1}$. Moreover, the action of $G$ on $W$ induces by composition an action of $G$ on every representation of $\GL(W)$ and in fact on every variety with an action of $\GL(W)$. For instance, the action of $G$ on $W$ induces an action of $G$ on the Grassmannian $\Gr(r,W)$ of $r$-dimensional linear subspaces of $W$, regarded in its Pl\"ucker embedding $\bbP \Lambda^r W$.

For a group $G$ acting on a space $W$, and $w \in W$, write $\Stab_G(w) = \{ g \in G : g \cdot w = w\}$ for the stabilizer subgroup of $w$ in $G$.
\subsection{Flattenings}\label{section: flattenings}
For every index set $I \subseteq [k]$, a tensor $T \in V_1 \ootimes V_k$ defines a linear map 
\[
T_I : \bigotimes_{i \in I} V_i^* \to \bigotimes_{j \notin I} V_j,
\]
called the $I$-th flattening map of $T$. Let $T_i$ denote the flattening map $T_{\{i\}}$. We say that a tensor~$T$ is \emph{concise} if the flattenings $T_i$ are injective for every $i$. The image $T_i(V_i^*) \subseteq \bigotimes_{j \neq i}V_j$ is a subspace of tensors of order $k-1$ of dimension at most $n_i$. This subspace uniquely determines $T$ up to the action of $\GL(V_i)$.

In \cite[Thm.~4.3]{CGZGap}, a characterization of degeneration in terms of degeneration of the flattening image was provided. A similar characterization in the case of tensor rank and border rank was also observed in \cite[Thm.~2.5]{BucLan} and \cite[Lemma 2.4]{GesOneVen}. Here, we provide a generalization of these results, which does not require a technical conciseness assumption. A similar generalization will appear in \cite{JJ26}. 

\begin{theorem}\label{thm: degenerations via grassmannians}
 Let $G$ be an algebraic group acting linearly on a vector space $W$ and let $A$ be a vector space. Let $T,S \in A \otimes W$ be two elements and let $E_T,E_S$ be the images of the flattening maps $T,S: A^* \to W$. Let $r = \dim E_T$ and regard $E_T$ as an element of $\Gr(r,W)$. The following are equivalent:
 \begin{enumerate}[\upshape(a)]
  \item $T$ degenerates to $S$ under the action of $(\GL(A) \times G)$ on $A \otimes W$;
  \item there exists a tensor $T' \in A \otimes W$ such that 
  \begin{itemize}
    \item $T$ degenerates to $T'$ under the action of $(\GL(A) \times G)$ on $A \otimes W$;
    \item the flattening $T' : A^* \to W$ has rank $r$;
    \item $T'$ restricts to $S$;
  \end{itemize}
  \item there is a linear space $E' \in \Gr(r,W)$ such that
  \begin{itemize}
    \item $E_T$ degenerates to $E'$ under the action of $G$ on $\Gr(r,W)$;
    \item $E_S \subseteq E'$.
  \end{itemize}
 \end{enumerate}
\end{theorem}
 \begin{proof}
    The proof is similar to the one of \cite[Thm.~4.3]{CGZGap}.

    We start by proving that (a) implies (c). Suppose $T$ degenerates to $S$ and let $g_\eps \in G$ and $x_\eps \in \GL(A)$ be curves such that 
    \[
    S = \lim_{\eps \to 0} (x_\eps \otimes g_\eps) T.
    \]
    The curve $g_\eps$ defines a curve $g_\eps E_T$ in $\Gr(r, W)$. Let $E' = \lim_{\eps \to 0} g_\eps E_T$. The same argument as in the first part of the proof of \cite[Thm.~4.3]{CGZGap} shows that $E_S \subseteq E'$. 
    
    The fact that (c) implies (b) is a direct consequence of \cite[Thm.~4.3]{CGZGap}. More precisely, let $T' \in A \otimes W$ be any tensor such that the image of $T' : A^* \to W$ is $E'$; in particular, such a flattening has rank equal to $r$. By \cite[Thm.~4.3]{CGZGap}, we deduce that $T$ degenerates to $T'$ under the action of $\GL(A) \times G$. The condition that $E_S \subseteq E'$ is equivalent to the fact that $T'$ restricts to $S$.

    The fact that (b) implies (a) is immediate: if $T$ degenerates to $T'$ and $T'$ restricts to $S$, then $T$ degenerates to $S$.
\end{proof}
The difference between \cref{thm: degenerations via grassmannians} and \cite[Thm.~4.3]{CGZGap} is that the latter has the additional assumption that $\dim E_S = \dim E_T = r$, in which case $E = E'$ and $S = T'$ up to the action of $\GL(A)$. In hindsight, the upgraded \cref{thm: degenerations via grassmannians} guarantees that any degeneration can be factored into a degeneration preserving conciseness, followed by a restriction: 
\begin{corollary}
Let $T,S \in V_1 \ootimes V_k$ be tensors such that $T$ is concise and degenerates to~$S$. Then there is a concise tensor $\tilde{T}$ such that $T$ degenerates to $\tilde{T}$ and $\tilde{T}$ restricts to $S$.
\end{corollary}
\begin{proof}
Let $m$ be the number of non-injective flattenings of $S$, that is 
\[
m = \# \{ j \in [k] : \rank( S_j : V_j^* \to \textstyle \bigotimes_{i \neq j} V_i) < \dim V_j\}.
\]
We proceed by induction on $m$. If $m = 0$, then $S$ is concise and there is nothing to prove.

If $m \geq 1$, assume the first flattening $S_1$ is not injective and apply \cref{thm: degenerations via grassmannians} with $A = V_1$ and $W = V_2 \ootimes V_k$ and $G = \GL(V_2) \ttimes \GL(V_k)$. There exists a tensor $T' \in V_1 \ootimes V_k$ such that $T' : V_1^* \to V_2 \ootimes V_k$ is injective, $T$ degenerates to $T'$ and $T'$ restricts to $S$. 

Since $T'$ restricts to $S$, the $j$-th flattening of $T'$ is injective whenever the $j$-th flattening of $S$ is injective. Moreover, by construction, the first flattening of $T'$ is injective, whereas the first flattening of $S$ is not. This shows that the number of non-injective flattenings of $T'$ is strictly smaller than that of $S$.

By the induction hypothesis, there exists a concise $\tilde{T}$ such that $T$ degenerates to $\tilde{T}$ and $\tilde{T}$ restricts to $T'$, hence to $S$. This concludes the proof.
\end{proof}

In the tensor setting, \cref{thm: degenerations via grassmannians} relates degeneration of tensors of order $k+1$ to degeneration of linear spaces of tensors of order $k$. The next result can be regarded as a \emph{local} version of \cref{thm: degenerations via grassmannians}, which relates degeneration of linear spaces with degeneration of elements in such spaces. 
\begin{proposition}\label{prop: degeneration in flattening image}
Let $G$ be an algebraic group acting linearly on a vector space $W$. Let $E,F \in \Gr(r, W)$ be linear spaces such that $E$ degenerates to $F$ under the action of $G$ on $\Gr(r,W)$. Suppose $\Stab_G(E)$ has a dense orbit in $E \subseteq W$. Then every $S \in F$ is a degeneration of the generic element $T \in E$: more precisely, there is a non-empty Zariski open subset $U \subseteq E$ such that, for every $S \in F$, and every $T \in U$, $S$ is a degeneration of $T$.
\end{proposition}
\begin{proof}
    Consider the projectivized tautological bundle $\calS \to \Gr(r,W)$, regarded as the subvariety 
    \[
    \calS = \{ (W', [T]) \in \Gr(r,W) \times \bbP W : [T] \in \bbP W' \} \subseteq \Gr(r,W) \times \bbP W;
    \]
    here $[T]$ denotes the class of $T \in W$ in $\bbP W$. Let $\pi_\Gr : \calS \to \Gr(r,W)$ and $\pi_W : \calS \to \bbP W$ be the two projections. The action of $G$ on $ \bbP W$ extends to an action to $\Gr(r,W)$ and diagonally to one on $\calS$: for $g \in G$ and $(W',[T]) \in \calS$, we have  
    \[
    g \cdot (W', [T]) = (gW',[gT]).
    \]
    For every $T \in E$, the orbit-closure $\bar{G \cdot (E,[T])}$ contains a point $(F, [S'])$ for some $S' \in F$, because $F \in \bar{G \cdot E}$. 

    Let $\calG = \bar{G \cdot \pi_{\Gr}^{-1}(E)}$, which is an irreducible subvariety of $\calS$ because it is the orbit-closure of $\{ E \} \times \bbP E \subseteq \calS$. Since $\Stab_G(E)$ has a dense orbit in $E$, we have 
    \[
   \calG = \bar{G \cdot \pi_{\Gr}^{-1}(E)} = \bar{ G \cdot \Stab_{G}(E) \cdot (E,[T])} = \bar{ G \cdot (E,[T])},
    \] 
    for a generic $T \in E$.

Moreover, for any $g \in G$, $\pi_W |_\calG ^{-1}( g E)$ is the entire $\bbP (gE)$. By semicontinuity of dimension of the fiber, the same holds at $F$. This shows that $(F,[S]) \in \calG$ for every $S \in F$ and therefore $(F,[S])$ is in the orbit-closure of $(E,[T])$. In particular $S$ is a degeneration of $T$.
\end{proof}

\subsection{Structure tensors of algebras}\label{subsec: structure tensors}

An algebra is a finite-dimensional vector space $\calA$ endowed with a bilinear multiplication map $\calA \times \calA \to \calA$, defining the \emph{structure tensor} $T_\calA \in \calA \otimes \calA^* \otimes \calA^*$. We do not require any property on the multiplication map: in particular, this framework includes associative algebras, skew algebras, Lie algebras and other algebraic structures. An algebra is unital if there exists an element $1_\calA \in \calA$ such that $T_\calA(1_\calA,-) = T_\calA(-,1_\calA) = \id_\calA$, when regarded as a map $\calA \to \calA$. Most of our results focus on the case of associative unital algebras. \cref{sec: lie sln} studies the Lie algebra case, in particular for $\calA = \fraksl_n$.

The $k$-fold product in an algebra $\calA$ is the multilinear map defined recursively as
\begin{equation}\label{eq: k-fold product}
\begin{aligned}
T^{(k)}_\calA : \underbrace{\calA \ttimes \calA}_k &\to \calA \\
(a_1 \vvirg a_k) &\mapsto T^{(k-1)}_\calA( T^{(2)}_{\calA}(a_1,a_2),a_3\vvirg  a_k )
\end{aligned}
\end{equation}
which defines a tensor $T^{(k)} _\calA \in \calA \otimes \calA^* \ootimes \calA^* $ of order $k+1$; we call it the $k$-th structure tensor of $\calA$. When $k=1$, we identify $T^{(1)}_{\calA}$ with the identity map $\calA \to \calA$ and, when $k=2$, $T^{(2)}_\calA = T_\calA$ is the structure tensor of $\calA$. In principle, if $\calA$ is not associative, then one may choose a different way of \emph{parenthesizing} the $k$-fold product in the definition of $T^{(k)}_\calA$. In \eqref{eq: k-fold product}, we consider the parenthesization from left to right, which is convenient to relate properties of $T^{(k)}_\calA$ for different values of $k$. We assume that $\calA$ is unital and associative for the rest of this section.

The group $U(\calA)$ of the units of $\calA$ naturally embeds in $\GL(\calA)$ in two ways, acting by left and by right multiplication in $\calA$. This induces two actions on $\calA$ and two actions on $\calA^*$, defined as follows for $u \in U(\calA)$, $a \in \calA$ and $\alpha \in \calA^*$:  
\[
\begin{array}{lll}
    u \cdot_L a = ua, & \qquad  &u \cdot_R a = a u^{-1},   \\
u \cdot_L \alpha = \alpha(u^{-1} \cdot - ), &\qquad &u \cdot_R \alpha = \alpha(- \cdot u);
\end{array}
\]
here $\alpha(u^{-1} \cdot - )$ and $\alpha(- \cdot u)$ indicate the linear maps obtained by precomposing $\alpha$ with the left multiplication by $u^{-1}$ or the right multiplication by $u$. In this way, the product $U(\calA)^{\times k+1}$ induces a \emph{sandwiching} action on $\calA \otimes \calA^* \ootimes \calA^*$; for any multilinear map $T: \calA \ttimes \calA \to \calA$, and any $(u_0 \vvirg u_k) \in U(\calA)^{\times k+1}$, define
\[
(u_0 \vvirg u_k) \cdot T :  (a_1 \vvirg a_k) \mapsto u_0 \cdot T( u_0^{-1} a_1 u_1 , u_1^{-1} a_2 u_2 \vvirg u_{k-1}^{-1} a_k u_{k})  \cdot u_{k}^{-1} .
\]
We can immediately show that this action stabilizes the structure tensors of $\calA$.
\begin{lemma}\label{lem: units stabilize}
Let $\calA$ be a unital associative algebra. For every $u_0 \vvirg u_k \in U(\calA)$, we have
\[
(u_0 \vvirg u_k) \cdot T_\calA^{(k)} = T_\calA^{(k)}.
\]
\end{lemma}
\begin{proof}
As a multilinear map, we have 
\[
T^{(k)}_{\calA}(a_1 \vvirg a_k) = a_1 \cdots a_k.
\]
By definition 
\begin{align*}
(u_0 \vvirg u_k) \cdot T_\calA^{(k)} (a_1 \vvirg a_k) = & u_0 \cdot T_\calA^{(k)} ( u_0^{-1} a_1 u_1  \vvirg u_{k-1}^{-1} a_k u_k)  \cdot u_{k}^{-1} =\\
& u_0 u_0^{-1} a_1 u_1 \cdots u_{k-1}^{-1} a_k u_k u_k^{-1} = a_1 \cdots a_k,
\end{align*}
where the last equality crucially uses the associativity of $\calA$. This shows that $T^{(k)}_{\calA}$ and $(u_0 \vvirg u_k) \cdot T^{(k)}_{\calA}$ coincide as multilinear maps. Therefore they are the same tensor, showing that $(u_0 \vvirg u_k)$ stabilizes $T^{(k)}_{\calA}$.
\end{proof}
Let $\calU$ be the image of  $U(\calA)^{\times k+1}$ in $\GL(\calA) \times \GL(\calA^*) \ttimes \GL(\calA^*)$ via the group homomorphism defining the sandwiching action. \cref{lem: units stabilize} shows that $\calU$ is a subgroup of the stabilizer of $T_\calA^{(k)}$ in $\GL(\calA) \times \GL(\calA^*) \ttimes \GL(\calA^*)$.

The following result shows that structure tensors of unital associative algebras satisfy the hypothesis required to apply \cref{prop: degeneration in flattening image}.
\begin{lemma}\label{lem: dense orbit for associative algebras}
Let $\calA$ be a unital associative algebra. Let $E = \image ( (T^{(k)}_\calA)_{k} : \calA \to \calA \otimes (\calA^*)^{\otimes (k-1)})$ be the image of the $k$-th flattening of $T_\calA^{(k)}$. Then $\calU$ stabilizes $E$ and has a dense orbit in $E$. 
\end{lemma}
\begin{proof}
    The statement that $\calU$ stabilizes $E$ is immediate because, by \cref{lem: units stabilize}, we have $\calU \subseteq \Stab_{G}(T_\calA^{(k)})$, where $G = \GL(\calA) \times \GL(\calA^*) \ttimes \GL(\calA^*)$.

    By construction, regarding the elements of $E$ as multilinear maps, we have 
    \[
    E = \{ T_b  : \calA ^{\times (k-1)} \to \calA: T_b(a_1 \vvirg a_{k-1}) = a_1 \cdots a_{k-1}b \text{ for some $b \in \calA$}\}. 
    \]
    The image in $\calU$ of the last copy $U(\calA) \subseteq U(\calA)^{\times (k+1)}$ acts by right multiplication on the last factor, hence all tensors $T_b$ where $b$ is a unit lie in the same $U(\calA)$-orbit. Since units are dense in $\calA$, this concludes the proof.
\end{proof}

By applying \cref{lem: dense orbit for associative algebras} and \cref{prop: degeneration in flattening image}, we obtain the following fundamental result: degeneration of higher order structure tensors propagates to lower order ones. 
\begin{theorem}\label{thm: degeneration between higher order algebras}
    Let $\calA,\calB$ be unital associative algebras. Let $k_0 \geq 2$ be an integer such that $T_\calA^{(k_0)}$ degenerates to $T_\calB^{(k_0)}$. Then $T_\calA^{(k)}$ degenerates to $T_\calB^{(k)}$ for every $k \leq k_0$.

   Moreover, if $\dim \calA = \dim \calB$, then $T_\calA^{(k)}$ degenerates to $T_\calB^{(k)}$ for every $k$.
\end{theorem}
\begin{proof}
For the first part, it suffices to show that the statement holds for $k = k_0-1$. Note that $T^{(k_0-1)}_\calA = T^{(k_0)}_{\calA}(- \vvirg -, 1_\calA)$ and $T^{(k_0-1)}_\calB = T^{(k_0)}_{\calB}(- \vvirg -, 1_\calB)$. 

 By \cref{prop: degeneration in flattening image}, $T^{(k_0-1)}_\calB$ is a degeneration of the generic element of the linear space $E = \image ( (T^{(k_0)}_\calA)_{k_0} : \calA \to \calA \otimes \calA^* \ootimes \calA^*)$. By \cref{lem: dense orbit for associative algebras}, $\calU$ has a dense orbit in $E$, hence any sufficiently generic element of $E$ is in the orbit of $T^{(k_0)}_\calA(- \vvirg - , 1_\calA) = T^{(k_0-1)}_\calA$. This concludes the proof of the first part.

 To prove the last claim, observe that, from the first part, if $T^{(k_0)}_\calA$ degenerates to $T^{(k_0)}_\calB$, then $T^{(2)}_\calA$ degenerates to $T^{(2)}_\calB$. Since $\dim \calA = \dim \calB$, by \cite[Thm. 3.2]{BlaLys}, we obtain that the degeneration of $T^{(2)}_\calA$ to $T^{(2)}_\calB$ can be achieved \emph{symmetrically}, in the following sense. There exists $g_\eps : \calA \to \calB$ such that 
 \[
T^{(2)}_\calB = \lim_{\eps \to 0} (g_\eps \otimes g_\eps^* \otimes g_\eps^*) T^{(2)}_\calA,
 \]
 where $g_\eps^* : \calA^* \to \calB^*$ is the inverse of the transpose of $g_\eps$. Explicitly, this means that, for every $b_1,b_2 \in \calB$, we have 
 \[
b_1 \cdot_B b_2 = \lim_{\eps \to 0} g_\eps( g_\eps^{-1}(b_1) \cdot _A g_\eps^{-1}(b_2)),
 \] 
 where $\cdot_A$ and $\cdot_B$ denote the multiplication in $\calA$ and $\calB$, respectively.
 We use induction on $k$ to show that $g_\eps \otimes (g_\eps^*)^{\otimes k}$ defines a degeneration of $T^{(k)}_\calA$ to $T^{(k)}_\calB$. The base of the induction for $k = 2$ is given by the argument above. For $k \geq 3$, and for every $b_1 \vvirg b_k \in \calB$, we have 
 \begin{align*}
\lim_{\eps \to 0} &g_\eps( (g_\eps^{-1}(b_1)) \cdots  (g_\eps^{-1}(b_k))) = \\ 
&\lim_{\eps \to 0} g_\eps \bigl(g^{-1}_\eps (g_\eps( g_\eps^{-1}(b_1) \cdots  g_\eps^{-1}(b_{k-1}))) \cdot (g^{-1}_\eps (b_k) ) \bigr) = \\
&\lim_{\eps \to 0} g_\eps ( g_\eps^{-1}( \lim_{\delta \to 0} [(g_\delta( g_\delta^{-1}(b_1) \cdots  g_\delta^{-1}(b_{k-1})))] ) \cdot(g^{-1}_\eps (b_k) ) )= \\
& \lim_{\eps \to 0} g_\eps ( g_\eps^{-1} (b_1 \cdots b_{k-1}) \cdot g_\eps^{-1}(b_k)) = b_1 \cdots b_k;
 \end{align*}
 the second equality holds because in the third line the arguments of both limits are polynomial functions in $\eps$ and $\delta$, and the limit reduces to the degree $0$ term in such a function, which coincides with the degree $0$ term in the second line; the third equality uses the induction hypothesis; the final equality holds because of the statement for $k=2$. This concludes the proof.
\end{proof}

\subsection{Several notions of rank and subrank}\label{subsec: subrank}
The \emph{unit tensor} of order $k$ and rank $r$ is the tensor 
\[
\bfu_k(r) = \sum_{i=1}^r e_i^{(1)} \ootimes e_i^{(k)} \in \bbC^r \ootimes \bbC^r
\]
where, for $p = 1 \vvirg k$, $\{ e_i^{(p)} : i \in [r]\}$ is a basis of the $p$-th copy of $\bbC^r$. Unit tensors defined with respect to different bases are isomorphic. 

The unit tensor is used to define several invariants of tensors. Let $T \in V_1 \ootimes V_k$ be a tensor. Then
\begin{itemize}
\item the \emph{rank} of $T$, denoted $\rmR(T)$, is the smallest $r$ such that $\bfu_k(r)$ restricts to~$T$;
\item the \emph{border rank} of $T$,  denoted $\uR(T)$, is the smallest $r$ such that $\bfu_k(r)$ degenerates to~$T$;
\item the \emph{subrank} of $T$, denoted $\rmQ(T)$, is the largest $q$ such that $T$ restricts to $\bfu_k(q)$;
\item the \emph{border subrank} of $T$, denoted $\uQ(T)$, is the largest $q$ such that $T$ degenerates to~$\bfu_k(q)$.
\end{itemize}
We immediately have 
\[
\rmQ(T) \leq \uQ(T) \leq \uR(T) \leq \rmR(T),
\]
and when $k = 2$, they all coincide with the matrix rank. 

The unit tensor $\bfu_3(r)$ is isomorphic to the structure tensor of the algebra $\bbC^r$ with coordinate-wise multiplication and $\bfu_{k+1}(r)$ is isomorphic to the structure tensor of the $k$-fold product in this algebra. In particular, \cref{thm: degeneration between higher order algebras} applies to the setting of rank and border rank, yielding the following result; see also \cite[Fact 4]{BlaMayShr23}.
\begin{proposition}\label{prop: R qnd Q as k grows}
Let $\calA$ be an associative unital algebra. Then, for every $k$,
\begin{itemize}
\item $\uR(T^{(k)}_{\calA}) \leq  \uR(T^{(k+1)}_{\calA})$,
\item $\uQ(T^{(k)}_{\calA}) \geq  \uQ(T^{(k+1)}_{\calA})$.
\end{itemize}
\end{proposition}
\begin{proof}
    For the first statement, let $r = \uR(T^{(k+1)}_{\calA})$. Then \cref{thm: degeneration between higher order algebras} with $\bbC^r$ playing the role of $\calA$ and $\calA$ playing the role of $\calB$ guarantees $\bfu_{k+1}(r) = T^{(k)}_{\bbC^r}$ degenerates to $T_\calA^{(k)}$, proving the first statement. The second statement is proved similarly.
\end{proof}
A consequence of \cref{prop: R qnd Q as k grows} is that, for every unital associative algebra $\calA$, the sequence of positive integers $q_k = \uQ(T^{(k)}_{\calA})$ stabilizes to a certain value $q_\infty$, only depending on the algebra $\calA$. We do not know exactly how this invariant relates to other invariants of the algebra $\calA$. If $\calA$ is commutative, the number of summands in its decomposition as sum of local algebras is a lower bound for $q_\infty$; moreover, \cref{prop: socle prop} will guarantee that if $\calA$ is local, then $q_\infty = 1$. The results of \cref{sec: MaMu} and \cref{sec: TMaMu} will show that $q_\infty = n$ for the algebras of $n \times n$ matrices and that of $n \times n$ upper triangular matrices. Another invariant of $\calA$ related to \cref{prop: R qnd Q as k grows} concerns \emph{how fast} $q_k$ stabilizes to $q_\infty$, that is the smallest $k_0$ such that $q_{k_0} = q_\infty$: we compute this value for the commutative algebras discussed in \cref{sec: commutative} and for the algebra of upper triangular matrices in \cref{sec: TMaMu}; we do not know its value for the algebra of $n \times n$ matrices.

The \emph{geometric rank} of $T\in V_1 \ootimes V_k$, denoted $\GR(T)$, is defined to be 
\[
\GR(T) = \codim \{ (\alpha_1 \vvirg \alpha_{k-1}) \in V_1 ^* \ttimes V_{k-1}^* : T(\alpha_1 \vvirg \alpha_{k-1}, \alpha_k) = 0 \text{ for all } \alpha_k \in V_k^*\}.
\]
Geometric rank was introduced in \cite{KopMosZui:GeomRankSubrankMaMu}. The subvariety 
\begin{equation}\label{eqn: geometric rank variety}
\calZ(T) = \{ (\alpha_1 \vvirg \alpha_{k-1}) \in V_1 ^* \ttimes V_{k-1}^* : T(\alpha_1 \vvirg \alpha_{k-1}, \alpha_k) = 0 \text{ for all } \alpha_k \in V_k^*\}    
\end{equation}
is the zero set of the system of multilinear equations $\image (T_k: V_k^* \to V_1 \ootimes V_{k-1})$, regarded as polynomials on (the cone over) the Segre variety $\bbP V_1^* \ttimes \bbP V_{k-1}^*$. It was proved in \cite[Thm.~3.2]{KopMosZui:GeomRankSubrankMaMu} that $\GR(T)$ can equivalently be defined by considering the system of equations given by any flattening $T_i : V_i^* \to \bigotimes_{j \neq i} V_j$; the varieties analogous to \eqref{eqn: geometric rank variety} all have the same codimension.

The theorem of semicontinuity of dimension of the fiber guarantees that geometric rank is semicontinuous under degeneration: if $T$ degenerates to $T'$, then $\GR(T) \geq \GR(T')$. Moreover, it is easy to see that $\GR(\bfu_k(r)) = r$. Therefore, geometric rank can be used to bound from above border subrank: 
\begin{equation}\label{eqn: Q bounded by Grank}
\uQ(T) \leq \GR(T).
\end{equation}
In the case of structure tensors of algebras, consider in \eqref{eqn: geometric rank variety} the variety defined with respect to the factor $\calA$; this is 
\[
\calZ_k(\calA) := \calZ(T^{(k)}_\calA) = \{ (a_1 \vvirg a_k) \in \calA \ttimes \calA : a_1 \cdots a_k = 0\}.
\]
 We obtain the following result, analogous to \cref{prop: R qnd Q as k grows}, but without the hypothesis that $\calA$ is unital or associative.
\begin{proposition}\label{prop: geoRank propagation}
Let $\calA$ be an algebra. Then for every $k$,
\[
\GR(T^{(k)}_{\calA}) \geq \GR(T^{(k+1)}_{\calA}).
\]
\end{proposition}
\begin{proof}
Clearly $\calZ_k(\calA) \times \calA \subseteq \calZ_{k+1}(\calA)$. We obtain 
\[
\codim_{\calA^{\times k}} \calZ_k(\calA) = \codim_{\calA^{\times k+1}} (\calZ_k(\calA) \times \calA) \geq \codim_{\calA^{\times k+1}} \calZ_{k+1}(\calA),
\]
The statement follows.
\end{proof}

Another notion of rank that will be useful in the following is the \emph{$G$-stable rank}, denoted $\rmR^G(T)$, which was introduced in \cite{Derk:GStableRank}. Its original definition is given in terms of the action of one-parameter subgroups of $\SL(V_1) \ttimes \SL(V_k)$ on the tensor product; it turns out that $\rmR^G(T)$ is a measure of the instability of $T$ under such action. We give here an equivalent definition, in terms of the solution of a linear program, as explained in \cite[Section 4B]{Derk:GStableRank}. Let $n_j = \dim V_j$ and let $\frakb = ( \{v_i^{(j)}\}_{i = 1 \vvirg n_j} : j = 1 \vvirg k)$ be a choice of bases for the vector spaces $V_1 \vvirg V_k$, so that $T$ has a unique expression in coordinates $T = \sum_{i_1 \vvirg i_k} T_{i_1 \vvirg i_k} v_{i_1}^{(1)} \ootimes v_{i_k}^{(k)}$; the \emph{support} of $T$ with respect to $\frakb$ is the set 
\[
\supp_\frakb (T) = \{ (i_1 \vvirg i_k): T_{i_1 \vvirg i_k} \neq 0\} \subseteq [n_1] \ttimes [n_k].
\]
Consider the value  
\[
\rmR^G_\frakb (T) = \min \left\{ \sum_{j=1}^k \sum_{i=1}^{n_j} x_{j,i} : x_{j,i} \geq 0, \sum_{j=1}^k x_{j,i_j} \geq 1 \text{ for all } (i_1 \vvirg i_k) \in \supp_\frakb(T) \right\};
\]
$\rmR^G_\frakb (T)$ is called T-stable rank in \cite[Sec.~4]{Derk:GStableRank}, where the `T' indicates the action of the torus in $\SL(V_1) \ttimes \SL(V_k)$ determined by the choice of bases $\frakb$. Using \cite[Cor. 4.2]{Derk:GStableRank}, one can define the $G$-stable rank of $T$ as
\[
\rmR^G (T) = \inf \{ \rmR^G_\frakb (T) : \frakb \text{ choice of bases of } V_1 \vvirg V_k\}.
\]
An important fact for our purposes is that for any tensor $T$, the border subrank is bounded from above by the $G$-stable rank; moreover, by definition, the $G$-stable rank is bounded above by the value $\rmR^G_\frakb (T)$ for any fixed $\frakb$. In summary, for every $T \in V_1 \ootimes V_k$, and every choice of bases $\frakb$, we have 
\begin{equation}\label{eqn: Q bounded by Gstable}
    \uQ(T) \leq \rmR^G(T) \leq \rmR^G_\frakb (T).
\end{equation}

\section{Commutative algebras}\label{sec: commutative}

In this section, $\calA$ is a unital commutative associative algebra, that is, an algebra that can be realized as a quotient of a polynomial ring by a suitable $0$-dimensional ideal. The structure tensors $T^{(2)}_\calA$ of such algebras have an algebraic characterization, see e.g. \cite[Sec.~2.1]{JLP}: $T \in \bbC^n \otimes \bbC^n \otimes \bbC^n$ is isomorphic to $T^{(2)}_\calA$ for some unital commutative algebra $\calA$ of dimension $n$ if and only if $T$ is \emph{binding} and $T$ satisfies Strassen's commutativity equations; see also \cite{LanMic}.

\subsection{Local algebras}
By the Chinese Remainder Theorem, every unital commutative algebra decomposes as 
\[
\calA \simeq \calA_1 \oplus \cdots \oplus \calA_m
\]
where $\calA_i$ are local algebras. In particular, the structure tensor decomposes as a direct sum of structure tensors of local algebras. As a consequence, border rank and border subrank of the structure tensors are, respectively, bounded from above and from below by the border ranks and border subranks of the structure tensors of the local summands:
\[
\uQ(T^{(k)}_\calA) \geq \textstyle \sum_{i=1}^m \uQ(T^{(k)}_{\calA_i}), \qquad  \uR(T^{(k)}_\calA) \leq \textstyle \sum_{i=1}^m \uR(T^{(k)}_{\calA_i}).
\]
To the extent of our knowledge, there are no known examples where these inequalities are strict. 

Understanding border rank of structure tensors of commutative algebras is considered a hard problem: the results of \cite{BlaLys} guarantee that it is at least as hard as deciding smoothability of $0$-dimensional schemes. This is considered a difficult problem, with connections to deep questions in algebraic geometry and the study of moduli spaces and singularities; we refer to \cite{Jeli20} for more details. An algebra of dimension $n$ is called \emph{smoothable} if $\uR(T^{(2)}_\calA) = n$ and in this case one can prove that $\uR(T^{(k)}_\calA) = n$ for every $k$. 

In this section, we record results on the structure tensors of some specific local commutative algebras. First, we provide a general upper bound on the border subrank, based on properties of the maximal ideal of the local algebra. Let $\calA$ be a local commutative algebra and let $\frakm$ be its (unique) maximal ideal: then, as a vector space $\calA = \langle 1 \rangle \oplus \frakm$ and $\frakm$ is \emph{nilpotent}, in the sense that $\frakm^s = 0$ for $s$ large enough. The \emph{socle degree} of $\calA$ is the largest $s$ for which $\frakm^s$ is nonzero; clearly $s \leq \dim \calA - 1$. The following result shows that for large $k$, the border subrank of $T^{(k)}_\calA$ is as small as possible. 
\begin{proposition}\label{prop: socle prop}
Let $\calA$ be a local commutative algebra with socle degree $s$ and let $r = \dim \frakm^s$.
 \begin{itemize}
\item if $k \geq 2s+1$ or if $(k,r) = (2s,1)$, then $\uQ(T^{(k)}_\calA) = 1$;
\item if $k = 2s$, then $\uQ(T^{(k)}_\calA) \leq 2$.
\end{itemize}
\end{proposition}
\begin{proof}
We describe the strategy of the proof, which is implemented in slightly different ways depending on the values of $k,r$. Let $q$ be an integer. To prove an upper bound $\uQ(T^{(k)}_\calA) \leq q$, we prove that all restrictions of $T^{(k)}_\calA$ to the tensor space $(\bbC^{q+1})^{\otimes (k+1)}$ are unstable for the action of $\SL_{q+1}^{\times (k+1)}$, in the sense that the closure of the $(\SL_{q+1}^{\times (k+1)})$-orbit of such restrictions contains the tensor $0$. Instability is a closed condition. Since the unit tensor $\bfu_{k+1}(q+1)$ is not unstable (see, e.g., \cite[Cor. 4.9]{BurIke}), we obtain an obstruction to the degeneration of $T^{(k)}_\calA$ to $\bfu_{k+1}(q+1)$, hence the desired upper bound.

In the case $k \geq 2s+1$, we want to prove the upper bound $\uQ(T^{(k)}_{\calA}) \leq 1$. Let $T' = X_0 \ootimes X_k (T^{(k)}_{\calA})$ be a generic restriction of $T^{(k)}_{\calA}$ to $\bbC^2 \ootimes \bbC^2$. Let $e_0,e_1$ be a basis of $\bbC^2$; without loss of generality, we may assume $X_p(1^*) = e_0$ for all $p= 1 \vvirg k$. This implies that, writing $T'$ as a linear combination of the basis elements $e_{i_0} \ootimes e_{i_k}$, every term in $T'$ has at most $s$ indices among $i_1 \vvirg i_k$ equal to $1$: this is because terms with nonzero coefficient $e_1$ can only appear from the restriction of elements in $\langle 1 \rangle^\perp \simeq \frakm^* \subseteq \calA^*$, and by definition $\frakm^{s+1} = 0$. Consider the one-parameter subgroup $g_\eps \subseteq \SL_2$ given by 
\[
g_\eps (e_0) = \eps e_0, \qquad g_\eps(e_1) = \eps^{-1} e_1.
\]
Then $\lim_{\eps \to 0} (\id_{\bbC^2} \otimes g_\eps^{\otimes k})T' = 0$. Indeed, by construction, every summand in $T'$ has at most $s$ factors $e_1$, so at least $k-s$ factors $e_0$. Since $k \geq 2s+1$, we have $k-s > s$, and each term of $T'$ is rescaled by a positive power of $\eps$; in particular it converges to $0$ as $\eps \to 0$. This guarantees that $T'$ is unstable and proves the desired upper bound.

The proof in the case $k = 2s$ with $r \geq 2$ is similar. Let $T' = X_0 \ootimes X_k (T^{(k)}_{\calA})$ be a generic restriction to $\bbC^3 \ootimes \bbC^3$. Let $e_0,e_1,e_2$ be a basis of $\bbC^3$; as before, without loss of generality, we may assume $X_p(1^*) = e_0$ for all $p= 1 \vvirg k$. Consider the one-parameter subgroups $g_\eps \subseteq \SL_3$ given by 
\[
g_\eps (e_0) = \eps^2 e_0, \qquad g_\eps(e_i) = \eps^{-1} e_i \text{ if $i=1,2$}.
\]
Then $\lim_{\eps \to 0} (\id_{\bbC^3} \otimes g_\eps^{\otimes k})T' = 0$. Indeed, every summand in $T'$ has at most $s$ factors in $\langle e_1,e_2\rangle$, so at least $k-s$ factors $e_0$. Since $k =2s$, we have $k-s = s$ and, since $r \geq 2$, each term is rescaled by a power of $\eps$ at least as large as $2s-s = s > 0$; in particular it converges to $0$ as $\eps \to 0$. This guarantees that $T'$ is unstable and proves the desired upper bound.

In the case $(k,r) = (2s,1)$, we proceed in a similar way.  Let $T' = X_0 \ootimes X_k (T^{(k)}_{\calA})$ be a generic restriction to $\bbC^2 \ootimes \bbC^2$. Let $e_0,e_1$ be a basis of $\bbC^2$; as before, without loss of generality, we may assume $X_p(1^*) = e_0$ for all $p= 1 \vvirg k$. We may further assume $X_0$ maps the $1$-dimensional space $\frakm^s$ to $\langle e_0 \rangle$. Consider the one-parameter subgroups $g_\eps \subseteq \SL_2$ given by 
\[
g_\eps (e_0) = \eps e_0, \qquad g_\eps(e_1) = \eps^{-1} e_1.
\]
We prove $\lim_{\eps \to 0} (g_\eps^{\otimes (k+1)})T' = 0$. Every summand in $T'$ has at most $s$ factors $e_1$ among the factors $1 \vvirg k$; moreover, the normalization on the $0$-th factor implies that if there are exactly $s$ factors $e_1$ in $1 \vvirg k$, then the $0$-th factor is $e_0$. We conclude, because terms with fewer than $s$ factors $e_1$ in $1 \vvirg k$ have at least $s+1$ factors $e_0$, therefore are rescaled by a positive power of $\eps$; terms with exactly $s$ factors $e_1$ in $\{1 \vvirg k\}$ have $s+1$ factors $e_0$, considering also the $0$-th factor. This guarantees the limit is $0$, and $T'$ is unstable.  
\end{proof}

The range of \cref{prop: socle prop} is sharp in some restricted cases: for instance, \cref{prop: null algebra} below guarantees that for $s=1$, $r \geq 2$, the upper bounds $\uQ(T^{(2)}_\calA) \leq 2$ cannot be improved. We do not know if the range $k \geq 2s+1$ in which $\uQ(T^{(k)}_\calA) =1$ can be improved, for instance to a range of the form $k \geq s+p$ for some constant $p$. We do not expect however that it should be possible to improve the range only relying on an instability argument: for instance, \cref{rmk: border subrank TRd3 via invariant} shows a structure tensor $T^{(k)}_{\calA}$ whose generic restriction to $\bbC^2 \ootimes \bbC^2$ is semistable, but there is no degeneration $T^{(k)}_{\calA}$ to $\bfu_{k+1}(2)$.

In the following \cref{sec: null algebra}, \cref{sec: truncated poly}, and \cref{sec: higher CW}, we determine upper and lower bounds on the border subrank, as well as the exact value of the geometric rank, for three classes of local algebras: null algebras $\calN_n = \bbC[x_1 \vvirg x_n]/(x_1 \vvirg x_n)^2$, truncated polynomial algebras $\calR_d = \bbC[x]/(x^d)$ and apolar algebras of homogeneous quadrics $\calQ_n = \bbC[x_1 \vvirg x_n]/ I$ where $I = (x_ix_j : i \neq j) + (x_i^2 - x_1^2: i = 2 \vvirg n)$. It is known that these three algebras are smoothable, so the border rank of the structure tensors coincides with the dimension of the algebras.

\subsection{Null algebras}\label{sec: null algebra}

Let $\mathcal{N}_n = \bbC[x_1,\dots,x_n]/(x_1,\dots,x_n)^2$ be the trivial unital algebra of dimension $n+1$. Fix the basis $1,x_1 \vvirg x_n$ of $\calN_{n}$ and let $1^*,x_1^*,\dots,x_n^*$ be its dual basis in $\calN_{n}^*$. Let $\frakm = \langle x_1 \vvirg x_n\rangle$ be the maximal ideal of $\calN_n$. The structure tensor of $k$-fold multiplication is 
    \begin{align*}
        T^{(k)}_{\calN_n} & = 1 \otimes (1^*)^{\otimes k} + \sum_{j = 1}^n [x_j \otimes (x_j^* \otimes 1^* \ootimes 1^* + \cdots + 1^* \ootimes 1^* \otimes x_j^*) ] 
    \end{align*}
where within each bracket one has the sum of the permutations of the tensor product of $k-1$ copies of $1^*$ and one copy of $x_j^*$.

It was shown in \cite[Lemma 3.4]{GesZuiNext}, following \cite[Lemma 3.5]{BlaLys}, that any binding tensor in $\bbC^{n+1} \otimes \bbC^{n+1} \otimes \bbC^{n+1}$ degenerates to $T^{(2)}_{\calN_n}$ or one of its permutations. Similarly, for every algebra $\calA$ of dimension $n+1$, the $k$-fold product tensor $T^{(k)}_\calA$ degenerates to $T^{(k)}_{\calN_n}$. So the structure tensors of $\calN_n$ are, in a way, minimal among all structure tensors of algebras of the same dimension. We determine the geometric rank and the border subrank of these tensors.

\begin{proposition}\label{prop: null algebra}
For every $k$ and every $n$, we have 
\begin{itemize}
\item $\GR(T^{(k)}_{\calN_n}) = 2$,
\item $\rmQ(T^{(k)}_{\calN_n}) = \uQ(T^{(k)}_{\calN_n}) = 2$ if $k = 2, n \geq 2$ and $\rmQ(T^{(k)}_{\calN_n}) = \uQ(T^{(k)}_{\calN_n}) = 1$ if $k \geq 3$.
\end{itemize}
\end{proposition}    

\begin{proof}
The irreducible components of the variety $\calZ_k(\calN_n) \subseteq \calN_{n}^{\times k}$ are the permutations of the products of linear spaces 
\[
\frakm \times \frakm \times \calN_{n} \times \cdots \times \calN_{n} \quad  \text{ or } \quad \{ 0 \} \times \calN_{n} \times \cdots \times \calN_{n},
\]
which have codimension $2$ and $n+1$ respectively. Indeed, suppose $(a_1 \vvirg a_k)$ does not belong to a product of the first type. Then $a_j$ is invertible for at least $k-1$ indices $j$, because every element in $\calN_n \setminus \frakm$ is invertible. The condition $a_1 \cdots a_k = 0$ would then imply that the unique non-invertible element among $a_1 \vvirg a_k$ is zero, therefore $(a_1 \vvirg a_k)$ belongs to a component of the second type. This shows $\GR(T^{(k)}_{\calN_n}) = 2$, which by \eqref{eqn: Q bounded by Grank} yields the bound $\uQ(T^{(k)}_{\calN_n}) \leq 2$ for every $k$.

If $k=2$, the lower bound $\rmQ(T^{(2)}_{\calN_n}) \geq 2$ when $n \geq 2$ follows from \cite[Thm.~1.10]{CGZGap} since $T^{(2)}_{\calN_n}$ is concise. If $n = 1$, then $T^{(k)}_{\calN_1}$ is isomorphic to the $W$-tensor on $k+1$ factors and we have $\rmQ(T^{(k)}_{\calN_1}) = \uQ(T^{(k)}_{\calN_1}) = 1$. The upper bound $\uQ(T^{(k)}_{\calN_n}) \leq 1$ when $k \geq 3$ follows from \cref{prop: socle prop}.
\end{proof}

\subsection{Truncated polynomial multiplication}\label{sec: truncated poly} 

Let $\mathcal{R}_{d} = \bbC[x]/(x^d)$. Consider the basis $1,x \vvirg x^{d-1}$ of $\calR_{d}$ and let $1^*,x^* \vvirg {x^{d-1}}^ *$ be its dual basis in $\calR_d^*$. Let $\frakm = (x) = \langle x \vvirg x^{d-1} \rangle$ be the maximal ideal of $\calR_d$, and let $\frakm^p = \langle x^p \vvirg x^{d-1} \rangle$ denote its $p$-th power. In particular, the socle degree of $\calR_d$ is $d-1$. The structure tensor of $k$-fold multiplication is 
    \begin{align*}
        T^{(k)}_{\calR_d} & = \sum_{\substack{0 \leq p_i \leq d-1 \\ p_1 + \cdots + p_k \leq d-1}} x^{p_1 + \cdots + p_k} \otimes {x^{p_1}}^* \ootimes {x^{p_k}}^*.
    \end{align*}
After fixing isomorphisms $\calR_d^* \simeq \bbC^d$ mapping ${x^i}^*$ to the basis vector $e_{i}$ and $\calR_d \simeq \bbC^d$ mapping $x^i$ to $e_{d-1-i}$, the tensor $T^{(k)}_{\calR_d}$ can be identified with
\[
T_{d,k} = \sum_{i_0 + \cdots + i_k = d-1} e_{i_0} \otimes e_{i_1} \ootimes e_{i_k}.
\]
Up to the action of the stabilizer in $\GL_d \ttimes \GL_d$ of the vector $e_0^{\otimes (k+1)}$, this is the generic element in the image of the $d$-th fundamental form of the Segre variety $\bbP^{d-1} \ttimes \bbP^{d-1}$ at $e_0^{\otimes (k+1)}$; see \cite[Sec.~12.2]{IveyLan}. The orbit closure of $T_{d,k}$ is called the $d$-th \emph{bud} of the Segre variety in \cite[Sec.~4]{LanMicBRAlg}, and sometimes it is referred to as the $(d-1)$-th curvilinear osculating variety of the Segre variety. In \cite{ChrGesStWer}, the tensors $T_{d,k}$ were used to construct augmentations of graph tensors in the framework of tensor networks; \cite[Prop.~2 (supp. material)]{ChrGesStWer} gives an explicit degeneration of $\bfu_{k+1}(d)$ to $T^{(k)}_{\calR_d}$.

The geometric rank of $T^{(k)}_{\calR_d}$ is easy to compute:
\begin{lemma}\label{lem: geometric rank TRd}
    For every $k$ and every $d$, we have 
 \[
 \GR(T^{(k)}_{\calR_d}) = d.
 \]
\end{lemma}
\begin{proof}
We prove that the irreducible components of the variety $\calZ_k(\calR_d) \subseteq \calR_{d}^{\times k}$ are the products of linear spaces
\[
\frakm^{p_1} \ttimes \frakm^{p_k} \quad \text{ with } p_1 + \cdots + p_k = d.
\]
To see this, suppose $(a_1 \vvirg a_k)$ does not belong to a product of this form. For every $j$, let $\ell_j$ be the \emph{leading exponent} of $a_j$, that is 
\[
\ell_j = \max \{ \ell \in \bbN : a_j \in \frakm^{\ell} \};
\]
equivalently $\ell_j$ is the smallest $\ell$ such that $a_j$ is linear combination of $x^{\ell} \vvirg x^{d-1}$. By the assumption, we have $\ell_1 + \cdots + \ell_k < d$; this is enough to guarantee $a_1 \cdots a_k \neq 0$ since its leading exponent is smaller than $d$. All linear spaces of the form above have codimension $d$, concluding the proof.
\end{proof}

The following result is the main theorem of this section, providing upper and lower bound on the border subrank of $T^{(k)}_{\calR_d}$.
\begin{theorem}\label{prop: lower bound QTRd}
Let $k,d \geq 1$ be integers. Then 
\[
\left\lfloor \frac{d-1}{k} \right\rfloor+1\leq\uQ(T^{(k)}_{\calR_d})\leq\left\lfloor\frac{2(d-1)}{k+1}\right\rfloor+1.
\]
\end{theorem}
\begin{proof}
The upper bound follows by \cref{thm: upper bound GstableRank TRd} below, which gives an upper bound on the $G$-stable rank, and therefore on the border subrank by \eqref{eqn: Q bounded by Gstable}. Indeed, \cref{thm: upper bound GstableRank TRd} guarantees 
\[
\uQ(T^{(k)}_{\calR_d}) \leq \left \lfloor \frac{(k+1)(q+1)(q+2)}{(k+1)(q+2)-r} \right\rfloor
\]
where $q = \left\lfloor\frac{2(d-1)}{k+1} \right\rfloor$ and $r = 2(d-1) - q(k+1)$; in particular $r \leq k$. We deduce that
\[
q+1\leq\frac{(k+1)(q+1)(q+2)}{(k+1)(q+2)-r}<q+2,
\]
and passing to the floor we get the desired upper bound.

The lower bound is obtained via an explicit construction of a degeneration of $T^{(k)}_{\calR_d}$ to the unit tensor $\bfu_{k+1}(q'+1)$, where $q' = \lfloor \frac{d-1}{k} \rfloor$. It suffices to provide the construction in the case $d = kq' + 1$ because $T^{(k)}_{\calR_d}$ restricts to $T^{(k)}_{\calR_{d'}}$ if $d' \leq d$. We are going to show that $T^{(k)}_{\calR_d}$ degenerates to the tensor
\begin{equation}\label{eqn: unit tensor in TRd}
\sum_{p=0}^{q'} x^{kp}\otimes(x^p)^*\otimes\cdots\otimes(x^p)^*,
\end{equation}
which is isomorphic to the unit tensor $\bfu_{k+1}(q'+1)$.

Let $\bar{p}=\lfloor q' / 2 \rfloor$. Apply a restriction to $T^{(k)}_{\calR_d}$ by mapping to zero basis vectors $(x^{p})^* \in \calR_d^*$ if $p > q'$ and $x^r \in \calR_d$ if $r$ is not a multiple of $k$. This results in a tensor 
\[
T = \sum_{\substack{0 \leq p_i \leq q' \\ k | (p_1 + \cdots + p_k)}} x^{p_1 + \cdots + p_k} \otimes {x^{p_1}}^* \ootimes {x^{p_k}}^* \in \bbC^{q'+1} \ootimes  \bbC^{q'+1}.
\]
Now define maps $X_0 \vvirg X_k: \bbC^{q'+1} \to \bbC^{q'+1}$, depending on a parameter $\eps$, as follows:
\[
   \begin{array}{rcl} 
    X_0 : \bbC^{q'+1} &\to & \bbC^{q'+1} \\
      x^{kp}& \mapsto &\eps^{-k\cdot2^{|p-\bar{p}|}}x^{kp} 
      \end{array}, \qquad
\begin{array}{rcl}
X_i : \bbC^{q'+1} &\to & \bbC^{q'+1} \\
   (x^p)^* &\mapsto &\eps^{2^{|p-\bar{p}|}}(x^p)^* 
\end{array} \text{ for } i = 1 \vvirg k.
\]
Then 
\begin{align*}
    (X_0 &\ootimes X_k)(T) \\[5pt]
    &=\sum_{\substack{0 \leq p_i \leq q \\ p_1 + \cdots + p_k = kp}} \eps^{-k\cdot2^{|p-\bar{p}|}+2^{|p_1-\bar{p}|}+\cdots+2^{|p_k-\bar{p}|}} \cdot x^{kp} \otimes {x^{p_1}}^* \ootimes {x^{p_k}}^*.
\end{align*}
We are going to prove that the exponents of $\eps$ in the summation above are strictly positive unless $p_1 = \cdots = p_k = p$, in which case they are $0$. This shows that the limit $\lim_{\eps \to 0} (X_0 \ootimes X_k)(T)$ is the sum of the terms with $p_1 = \cdots = p_k = p$, that is the unit tensor of \eqref{eqn: unit tensor in TRd}.  

For $p_1+\cdots+p_k=kp$, using the AM-GM inequality, and the triangular inequality, we obtain
\begin{align*}
    2^{|p_1-\bar{p}|}+\cdots+2^{|p_k-\bar{p}|}&\geq k\cdot\sqrt[k]{2^{|p_1-\bar{p}|}\cdot\cdots\cdot2^{|p_k-\bar{p}|}} \\
    &= k\cdot2^{(|p_1-\bar{p}|+\cdots+|p_k-\bar{p}|)/k} \\
    &\geq k\cdot2^{(|p_1+\cdots+p_k-k\bar{p}|)/k} \\
    &= k\cdot2^{|p-\bar{p}|}.
\end{align*}
This shows that the exponents of $\eps$ in the degeneration are always nonnegative. Moreover, they are zero if and only if both the triangular inequality and the AM-GM inequality are equalities. The AM-GM inequality holds with the equal sign if and only if $|p_j - \bar{p}|$ does not depend on $j$; the triangular inequality holds with the equal sign if and only if the differences $p_j - \bar{p}$ all have the same sign. In turn, this implies $p_1 = \cdots = p_k = p$ as desired. This shows that $T$ degenerates to $\bfu_{k+1}(q'+1)$, concluding the proof.
\end{proof}

\cref{thm: upper bound GstableRank TRd} below provides the upper bound on the G-stable rank required in the proof of \cref{prop: lower bound QTRd}. We point out that in the case of the algebra $\calR_d$, geometric rank does not provide any non-trivial upper bound on border subrank, as observed in \cref{lem: geometric rank TRd}. On the other hand the gap between the lower bound on the border subrank in \cref{prop: lower bound QTRd} and the upper bound obtained via G-stable rank is essentially a multiplicative factor of $2$, independent of $k$ and $d$. In \cref{rmk: sharp Gstable rank}, we provide evidence that the upper bound of \cref{thm: upper bound GstableRank TRd} on G-stable rank is sharp; however, the induced upper bound on the border subrank certainly is not, see \cref{rmk: border subrank TRd4 via invariant}.
\begin{theorem}\label{thm: upper bound GstableRank TRd}
Let $k,d \geq 1$ be integers, let $q=\left\lfloor\frac{2(d-1)}{k+1}\right\rfloor$ and $r = 2(d-1) - q(k+1)$. Then 
\[
\rmR^G(T^{(k)}_{\calR_d}) \leq \rmR^G_\frakb (T^{(k)}_{\calR_d}) \leq  \frac{(k+1)(q+1)(q+2)}{(k+1)(q+2)-r},
\]
where $\frakb$ is the choice of basis $1, x \vvirg x^{d-1}$ on $\calR_d$ and its dual basis on $\calR_d^*$.
\end{theorem}
\begin{proof}
After suitably relabeling the basis vectors, the tensor $T^{(k)}_{\calR_d}$ can be identified with
\[
T_{d,k} = \sum_{i_0 + \cdots + i_k = d-1} e_{i_0} \otimes e_{i_1} \ootimes e_{i_k}.
\]
We assign a weight $w(e_i)$ to each $e_{i}$ such that $w(e_{i_0})+\cdots+w(e_{i_k})\geq 1$ for $i_0 + \cdots + i_k = d-1$; the values $x_{ij} = w(e_i)$ for every $j$ then gives a feasible solution of the linear program in the definition of $\rmR^G_\frakb(T_{d,k})$, therefore $\displaystyle(k+1)(w(e_0) + \cdots + w(e_{d-1}))$ is an upper bound on the $G$-stable rank. Assign weight $w(e_i) = 0$ if $ i \geq q+1$ and 
\[
w(e_i)=     \frac{q+1-i}{(q+1)(k+1)-(d-1)} \text{ if }i \leq q . 
\]
For $i_0 + \cdots + i_k = d-1$, we may assume $i_j\leq q$ if $ j\leq s$ and $i_{j}\geq q+1$ for $j> s$, for some $s\in\{0,\cdots,k\}$. Then $i_0+\cdots+i_s\leq d-1-(q+1)(k-s)$. 
We obtain that the sum of the weights is 
\begin{align*}
    w(e_{i_0})+\cdots+w(e_{i_k})&=\sum_{j=0}^s\frac{q+1-i_j}{(q+1)(k+1)-(d-1)}\\
    &=\frac{(q+1)(s+1)-(i_0+\cdots+i_s)}{(q+1)(k+1)-(d-1)} \\
    &\geq\frac{(q+1)(s+1)-[d-1-(q+1)(k-s)]}{(q+1)(k+1)-(d-1)} \\
    &=\frac{(q+1)(k+1)-(d-1)}{(q+1)(k+1)-(d-1)}= 1.
\end{align*}
Then the total weight $(k+1) ( w(e_0) + \cdots + w(e_{d-1}))=\frac{(k+1)(q+1)(q+2)}{(k+1)(q+2)-r}$ is an upper bound on the $G$-stable rank of $T^{(k)}_{\calR_d}$.
\end{proof}

We expect the upper bound of \cref{thm: upper bound GstableRank TRd} to be sharp. We show that it is in a specific case, using a characterization of $G$-stable rank in terms of spectral norm of flattening maps in the orbit-closure of the tensor, see \cite[Thm. 5.2]{Derk:GStableRank}.
\begin{remark}\label{rmk: sharp Gstable rank}
Consider $k = 2$, $d = 4$. In this case, after relabeling the basis vectors, the tensor $T^{(2)}_{\calR_4}$ can be identified with
\begin{align*}
T = &e_0 \otimes e_0 \otimes e_3 + e_0 \otimes e_3 \otimes e_0  + e_3 \otimes e_0 \otimes e_0 + 
 e_1 \otimes e_1 \otimes e_1 + \\
 &e_0 \otimes e_1 \otimes e_2 + e_0 \otimes e_2 \otimes e_1 + e_1 \otimes e_0 \otimes e_2 + e_1 \otimes e_2 \otimes e_0 + e_2 \otimes e_0 \otimes e_1 + e_2 \otimes e_1 \otimes e_0.
\end{align*}
With the notation of \cref{thm: upper bound GstableRank TRd}, we have $q = 2$ and $r = 0$, therefore the upper bound on the $G$-stable rank is $\rmR^{G}(T^{(2)}_{\calR_4}) \leq 3$. We prove a matching lower bound using the characterization of \cite[Thm. 5.2]{Derk:GStableRank}. In particular, we have 
\[
\rmR^G(T) = \sup_{g_i \in \GL_4} \min_{s = 0,1,2}  \left( \frac{ {\bigl\Vert (g_0 \otimes g_1 \otimes g_2) T \bigr\Vert_2}}{ {\sigma_s ( (g_0 \otimes g_1 \otimes g_2) T)}} \right)^2,
\]
where $\Vert - \Vert_2$ indicates the Frobenius norm of the tensor and $\sigma_s(-)$ indicates the spectral norm of the $s$-th flattening. In particular, for any choice of $g_i \in \GL_4$, we obtain a lower bound on the $G$-stable rank. Define the linear maps $g_i$, depending on a parameter $\eps$, to be
    \[
    g_0 = g_1 = g_2 = \left( \begin{array}{cccc}
        1 & 0 & 0 & 0 \\
        0 & \eps & 0 & 0 \\
        0 & 0 & 1 & 0 \\     
        0 & 0 & 0 & \eps^2 
    \end{array}\right).   
    \]
Let $T_\eps = (g_0 \otimes g_1 \otimes g_2) T$. One can compute that 
\[
\Vert T_\eps \Vert_2^2 = 6\eps^2 + O(\eps^3).
\]
Moreover, by symmetry, the three flattenings coincide and the square of their spectral norm is $\sigma_i(T_\eps)^2 = 2 \eps^2 + O(\eps^3) $. Using \cite[Thm. 5.2]{Derk:GStableRank}, we obtain 
\[
\rmR^G(T^{(2)}_{\calR_4}) \geq \lim_{\eps \to 0} \frac{ 6\eps^2 + 3\eps^4 + \eps^6}{2 \eps^2} = 3.
\]
\end{remark}
In the case described in \cref{rmk: sharp Gstable rank}, we can prove that the upper bound $\uQ(T^{(2)}_{\calR_4}) \leq 3$ obtained via $G$-stable rank is not sharp. Indeed, we prove $\uQ(T^{(2)}_{\calR_4}) = 2$.
\begin{remark}\label{rmk: border subrank TRd4 via invariant}
Consider $k = 2$, $d = 4$. We provide the upper bound $\uQ(T^{(2)}_{\calR_4}) \leq 2$ which matches the lower bound of \cref{prop: lower bound QTRd} and improves on the upper bound obtained via $G$-stable rank:
\[
2 = \uQ(T^{(2)}_{\calR_4}) < \lfloor \rmR^G(T^{(2)}_{\calR_4}) \rfloor = 3.
\]
To prove $\uQ(T^{(2)}_{\calR_4}) \leq 2$ consider the subvariety 
\[
\calY = \{ T \in \bbC^3 \otimes \bbC^3 \otimes \bbC^3 : T \text{ is a degeneration of } T^{(2)}_{\calR_4}\}
\]
of degenerations of $T^{(2)}_{\calR_4}$ in $\bbC^3 \otimes \bbC^3 \otimes \bbC^3$. If $\uQ(T^{(2)}_{\calR_4}) =3$, then the unit tensor $\bfu_3(3)$ would belong to $\calY$. However, the polynomial $F_6^2 - F_{12}$ defined in terms of the fundamental invariants of $\bbC^3 \otimes \bbC^3 \otimes \bbC^3$ described in \cref{sec: invariants 333} vanishes on $\calY$ but not on $\bfu_3(3)$. This can be proved via an explicit calculation. This shows that $\bfu_3(3) \notin \calY$, therefore $\uQ(T^{(2)}_{\calR_4}) \leq 2$ as desired. We point out $\rmQ(T^{(2)}_{\calR_4}) = 2$ as well, by \cite[Thm. 1.10]{CGZGap}.
\end{remark}

We can provide a sharp upper bound also in the case $(k,d) = (3,3)$.
\begin{remark}\label{rmk: border subrank TRd3 via invariant}
Consider $k = 3$, $d = 3$. We provide the upper bound $\uQ(T^{(3)}_{\calR_3}) \leq 1$ which matches the lower bound of \cref{prop: lower bound QTRd} and improves on the upper bound obtained via $G$-stable rank:
\[
1 = \uQ(T^{(3)}_{\calR_3}) \leq \lfloor \rmR^G(T^{(3)}_{\calR_3}) \rfloor \leq 2.
\]
Unlike the case of \cref{rmk: border subrank TRd4 via invariant}, we do not know if the upper bound on the $\rmR^G(T^{(3)}_{\calR_3})$ in \cref{thm: upper bound GstableRank TRd} is sharp in this case. However, the proof of the upper bound $\uQ(T^{(3)}_{\calR_3}) \leq 1$ is similar to the one of \cref{rmk: border subrank TRd4 via invariant}. Consider the subvariety 
\[
\calY = \{ T \in \bbC^2 \otimes \bbC^2 \otimes \bbC^2 \otimes \bbC^2 : T \text{ is a degeneration of } T^{(3)}_{\calR_3}\}
\]
of degenerations of $T^{(3)}_{\calR_3}$ in ${\bbC^2}^{\otimes 4}$. If $\uQ(T^{(3)}_{\calR_3}) =2$, then the unit tensor $\bfu_4(2)$ would belong to $\calY$. However, the polynomial $16 F_2^3 - 27 F_{6}$ defined in terms of the fundamental invariants of ${\bbC^2}^{\otimes 4}$ described in \cref{sec: invariants 2222} vanishes on $\calY$ but not on $\bfu_4(2)$. This can be proved via an explicit calculation. This shows that $\bfu_4(2) \notin \calY$, therefore $\uQ(T^{(3)}_{\calR_3}) \leq 1$ as desired. Note that the socle degree of $\calR_3$ is $2$, therefore the upper bound $\uQ(T^{(3)}_{\calR_3}) \leq 1$ does not follow from \cref{prop: socle prop}: in fact, the generic restriction of $T^{(3)}_{\calR_3}$ to $\bbC^2 \otimes \bbC^2 \otimes \bbC^2 \otimes \bbC^2$ is semistable, therefore an instability argument such as the one of \cref{prop: socle prop} cannot be applied.
\end{remark}

\subsection{Generalized Coppersmith--Winograd tensors}\label{sec: higher CW}

Consider the algebra 
\[
\calQ_n = \bbC[x_1 \vvirg x_n]/ \bigl( (x_ix_j : i \neq j) + (x_i^2 - x_1^2 ) \bigr).
\]
In the language of apolarity, $\calQ_n$ is the \emph{apolar algebra} of the quadric $g = z_1^2 + \cdots + z_n^2$, see \cite{HJMS22}. In particular $\calQ_n$ is local, with maximal ideal $\frakm = (x_1 \vvirg x_n)$; the square of the maximal ideal $\frakm^2$ is generated by a single element, say, $y$, spanning the quotient of the degree two component of the polynomial ring. In particular, the socle degree of $\calQ_n$ is $2$ for every $n \geq 1$. Let $1^*, x_1^* \vvirg x_n^*, y^*$ be the basis of $\calQ_n^*$ dual to the basis $1, x_1 \vvirg x_n, y$ of $\calQ_n$. We have $\dim \calQ_n = n+2$.

The tensor $T^{(2)}_{\calQ_n}$ is isomorphic to the big Coppersmith--Winograd tensor in $\bbC^{n+2} \otimes \bbC^{n+2} \otimes \bbC^{n+2}$ defined as
\[
\sum_{i =1}^n (e_0 \otimes e_i \otimes e_i + e_i \otimes e_0 \otimes e_i + e_i \otimes e_i \otimes e_0) + e_0 \otimes e_0 \otimes e_{n+1} + e_0 \otimes e_{n+1} \otimes e_0 + e_{n+1} \otimes e_0 \otimes e_0,
\]
which was introduced in \cite{CopWin} and it has been used in the context of Strassen's laser method to determine upper bounds on the exponent of matrix multiplication.

For the algebra $\calQ_n$, we can completely characterize the geometric rank and the border subrank of the $k$-fold multiplication. The first result shows that geometric rank is always equal to $3$, as in the case of the big Coppersmith--Winograd tensor:
\begin{proposition}\label{prop: GR any CW}
For every $n \geq 1$ and every $k \geq 2$, we have
\[
\GR(T^{(k)}_{\calQ_n}) = 3.
\]
\end{proposition}
\begin{proof}
    Consider the variety $\calZ_k(\calQ_n) \subseteq \calQ_n^{\times k}$. Let $\frakm_1 = \langle x_1 \vvirg x_n \rangle$ be the degree $1$ component of $\calQ_n$. The bilinear form  
    \[
f^* : \frakm_1 \times \frakm_1 \to \langle y \rangle \simeq \bbC
    \]
    defined by the multiplication in $\calQ_n$ is nondegenerate. It defines components of codimension $3$ in $\calZ_k(\calQ_n)$ given by all permutations of 
    \[
   \calZ' \times \calQ_n^{\times k-2} \subseteq  \calQ_n^{\times k}
    \]
    where $\calZ' \subseteq \calQ_n^{\times 2}$ is the variety defined by $f^*(a_1,a_2) = 0$ on $\frakm \times \frakm$, which has codimension $1$ in $\frakm \times \frakm$, hence codimension $3$ in $\calQ_n \times \calQ_n$ since $\calQ_n = \langle 1 \rangle \oplus \frakm$.    This shows $\GR(T^{(k)}_{\calQ_n}) \leq 3$. To show the lower bound, observe that if $(a_1 \vvirg a_k)$ satisfies $1^*(a_j) \neq 0$ for all but at most one index $j_0$, then either $a_{j_0} =0$ or $a_1 \cdots a_k \neq 0$. Therefore, the components of $\calZ_k(\calQ_n)$ are contained in the product $\frakm \times \frakm \times \calQ_n^{\times (k-2)}$ or one of its permutations, or they coincide with the trivial components $\{0\} \times \calQ_n^{\times (k-1)}$ or one of its permutations. The latter have codimension equal to $\dim \calQ_n  = n+2$; the former have codimension at least $1$ in $\frakm \times \frakm \times \calQ_n^{\times (k-2)}$ because the generic element of this product does not belong to $\calZ_k(\calQ_n)$. This concludes the proof that $\GR(T^{(k)}_{\calQ_n}) \geq 3$. 
\end{proof}

The characterization of border subrank is as follows.
\begin{theorem}\label{thm: subrank higher CW}
Let $n \geq 1$ and $k \geq 2$. Then 
\begin{enumerate}
    \item if $k \geq 4$ or if $(k,n) = (3,1)$ then $\rmQ(T^{(k)}_{\calQ_n}) = \uQ(T^{(k)}_{\calQ_n}) = 1$;
    \item if $k = 3$ and $n \geq 2$, or if $(k,n) = (2,1), (2,2)$ then $\rmQ(T^{(k)}_{\calQ_n}) =\uQ(T^{(k)}_{\calQ_n}) = 2$;
    \item if $k = 2$ and $n \geq 3$ then $\rmQ(T^{(k)}_{\calQ_n}) =\uQ(T^{(k)}_{\calQ_n}) = 3$.
\end{enumerate}
\end{theorem}

We prove the different cases separately in the following.
\begin{proposition}\label{prop: border subrank CW k2}
Let $k = 2$. Then 
\begin{align*}
&\rmQ(T^{(2)}_{\calQ_1}) = \uQ(T^{(2)}_{\calQ_1}) = 2, \\
&\rmQ(T^{(2)}_{\calQ_2}) = \uQ(T^{(2)}_{\calQ_2}) = 2, \\
&\rmQ(T^{(2)}_{\calQ_n}) = \uQ(T^{(2)}_{\calQ_n}) = 3 \quad \text{ for $n \geq 3$}.
\end{align*}
\end{proposition}
\begin{proof}
Since $T^{(2)}_{\calQ_n}$ is a restriction of $T^{(2)}_{\calQ_{n+1}}$, we have $\rmQ(T^{(2)}_{\calQ_n}) \leq \rmQ(T^{(2)}_{\calQ_{n+1}})$ and $\uQ(T^{(2)}_{\calQ_n}) \leq \uQ(T^{(2)}_{\calQ_{n+1}})$. The lower bound $\rmQ(T^{(2)}_{\calQ_1}) \geq 2$ follows from \cite[Thm.~1.10]{CGZGap} because $T^{(2)}_{\calQ_1}$ is concise. The upper bound $\uQ(T^{(2)}_{\calQ_n}) \leq 3$ follows from \cref{prop: GR any CW}. To conclude, it remains to prove the upper bound $\uQ(T^{(2)}_{\calQ_2}) \leq 2$ and the lower $\rmQ(T^{(2)}_{\calQ_3}) \geq 3$.

    For $n=2$, consider the variety
    \[
   \calC = \{ T \in \bbC^3 \otimes \bbC^3 \otimes \bbC^3: T \text{ is a degeneration of } T^{(2)}_{\calQ_2}\}.
    \]
    A direct computation shows the polynomial $F_6^2 - 4F_{12}$, defined in terms of the invariants described in \cref{sec: invariants 333}, vanishes on $\calC$. Since $F_6^2- 4F_{12}$ does not vanish on $\bfu_3(3)$, this shows the upper bound $\uQ(T^{(2)}_{\calQ_2}) \leq 2$, which in turn implies $\rmQ(T^{(2)}_{\calQ_n}) =\uQ(T^{(2)}_{\calQ_n}) = 2$ for $n = 1,2$.

    Next, we show that $T^{(2)}_{\calQ_3}$ restricts to $\bfu_3(3)$, providing $\rmQ(T^{(2)}_{\calQ_3}) = 3$ and therefore $\rmQ(T^{(2)}_{\calQ_n}) =3$ for all $n \geq 3$. Consider the restriction given by the following maps:
    \[
    \begin{array}{rlll}
    X_0 : &\calA & \to &\bbC^3 \\
        &1 & \mapsto & 0 \\
        &x_1 & \mapsto & e_0 \\ 
        &x_2 &\mapsto & e_1 \\ 
        &x_3 &\mapsto & 0 \\ 
        &y & \mapsto & e_2 
    \end{array},
    \qquad
        \begin{array}{rlll}
    X_1 : &\calA^* & \to &\bbC^3 \\
        &1^* & \mapsto & e_1 \\
        &x_1^* & \mapsto & e_0 \\ 
        &x_2^* &\mapsto & 0 \\ 
        &x_3^* &\mapsto & e_2 \\ 
        &y^* & \mapsto & 0
    \end{array},
    \qquad
    \begin{array}{rlll}
    X_2 : &\calA^* & \to &\bbC^3 \\
        &1^* & \mapsto & e_0 \\
        &x_1^* & \mapsto & 0 \\ 
        &x_2^* &\mapsto & e_1 \\ 
        &x_3^* &\mapsto & e_2 \\ 
        &y^* & \mapsto & 0 
    \end{array}.
    \]
    One can immediately verify that $(X_0 \otimes X_1 \otimes X_2)(T^{(2)}_{\calQ_3}) = \bfu_3(3)$.
\end{proof}

\begin{proposition}\label{border subrank CW k3}
If $k = 3$, then 
\[
\rmQ(T^{(3)}_{\calQ_1}) = \uQ(T^{(3)}_{\calQ_1}) = 1, \quad \text{ and } \quad \rmQ(T^{(3)}_{\calQ_n}) = \uQ(T^{(3)}_{\calQ_n}) = 2 \text{ for every } n \geq 2.
\]
\end{proposition}
\begin{proof}
We proceed in a way similar to \cref{prop: border subrank CW k2}. Clearly $\rmQ(T^{(3)}_{\calQ_n}) \geq 1$ for every $n$. Consider the variety
    \[
   \calD_n = \{ T \in \bbC^2 \otimes \bbC^2 \otimes \bbC^2 \otimes \bbC^2 : T \text{ is a degeneration of } T^{(3)}_{\calQ_n}\}.
    \]
    If $n = 1$, a direct computation shows the polynomial $2F_2^3-27F_6$, defined in terms of the invariants described in \cref{sec: invariants 2222}, vanishes on $\calD_1$. Since it does not vanish on $\bfu_4(2)$, we obtain the upper bound $\uQ(T^{(3)}_{\calQ_1}) \leq 1$. This proves $\rmQ(T^{(3)}_{\calQ_1}) =\uQ(T^{(3)}_{\calQ_1}) = 1$.

    Next, we show that $T^{(3)}_{\calQ_2}$ restricts to $\bfu_4(2)$, providing $\rmQ(T^{(3)}_{\calQ_2}) \geq 2$ and therefore $\rmQ(T^{(3)}_{\calQ_n}) \geq \uQ(T^{(3)}_{\calQ_n}) \geq 2$ for all $n \geq 3$. Consider the restriction given by the following maps:
    \[
    \begin{array}{rlll}
    X_0 : &\calA & \to &\bbC^2 \\
        &1 & \mapsto & 0 \\
        &x_1 & \mapsto & e_0 \\ 
        &x_2 &\mapsto & 0 \\ 
        &y & \mapsto & e_1 
    \end{array},
    \qquad
        \begin{array}{rlll}
    X_1 : &\calA^* & \to &\bbC^2 \\
        &1^* & \mapsto & e_1 \\
        &x_1^* & \mapsto & e_0 \\ 
        &x_2^* &\mapsto & 0 \\ 
        &y^* & \mapsto & 0
    \end{array},
    \qquad
    \begin{array}{rlll}
    X_2 , X_3 : &\calA^* & \to &\bbC^2 \\
        &1^* & \mapsto & e_0 \\
        &x_1^* & \mapsto & 0 \\ 
        &x_2^* &\mapsto & e_1 \\ 
        &y^* & \mapsto & 0 
    \end{array}.
    \]
    One can immediately verify that $(X_0 \ootimes X_3 )(T^{(3)}_{\calQ_2}) = \bfu_4(2)$.

To prove the upper bound $\uQ(T^{(3)}_{\calQ_n}) \leq 2$, we show that if $n \geq 2$, the variety $\calD_n$ does not depend on $n$ and it is a hypersurface not containing some tensors of border rank $3$. In fact, $\calD_n$ is the hypersurface cut out by the hyperdeterminant in $\bbC^2 \ootimes \bbC^2$. To see this, let $X_0 \vvirg X_3$ be linear maps and let $T = (X_0 \ootimes X_3) (T^{(3)}_{\calQ_n})$ be a generic point of $\calD_n$. After possibly acting with $\GL_2 \ttimes \GL_2$, we may assume 
    \begin{align*}
\begin{array}{rlcl}
    X_0 : & 1 & \mapsto &e^{(0)}_0 \\
          & y & \mapsto &e^{(0)}_1 \\
          & x_j & \mapsto &v^{(0)}_j
    \end{array}
\qquad \begin{array}{rlcl}
    X_i : & 1^* & \mapsto &e^{(i)}_0 \\
          & y^* & \mapsto &e^{(i)}_1 \\
          & x_j^* & \mapsto &v^{(i)}_j
    \end{array} \text{ for $i = 1,2,3$}
\end{align*}
for generic vectors $v^{(i)}_j \in \bbC^2$. We prove that for every $n \geq 2$, the tangent space of $\calD_n$ at $T$ is 
\[
\bfT_{T} \calD_n = \langle e_0^{(0)*} \otimes e_1^{(1) *} \otimes e_1^{(2) *} \otimes e_1^{(3) *} \rangle^\perp .
\]
Such tangent space is the image of the differential at the point $(X_0,X_1,X_2,X_3)$ of the restriction map 
\[
\Phi : \Hom(\calQ_n, \bbC^2) \times \Hom(\calQ_n^*, \bbC^2) \times \Hom(\calQ_n^*, \bbC^2) \times \Hom(\calQ_n^*, \bbC^2)\to \bbC^2 \otimes \bbC^2 \otimes \bbC^2 \otimes \bbC^2
\]
defined by $\Phi(Y_0 \vvirg Y_3) = Y_0 \ootimes Y_3 ( T^{(3)}_{\calQ_n})$. By Leibniz rule, we have that
\begin{align*}
\bfT_T \calD_n = &\image ( \rmd_{(X_0 \vvirg X_3)} \Phi) = \\ &\Hom(\calQ_n,\bbC^2) \cdot ((X_1 \otimes X_2 \otimes X_3) (T^{(3)}_{\calQ_n})) + \cdots + \Hom(\calQ_n^*,\bbC^2)  (X_0 \otimes X_1 \otimes X_2 )(T^{(3)}_{\calQ_n})
\end{align*}
is the sum of four linear spaces. We show that all four belong to $\langle e_0^{(0)*} \otimes e_1^{(1) *} \otimes e_1^{(2) *} \otimes e_1^{(3) *} \rangle^\perp $.

Let $Z \in \Hom(\calQ_n,\bbC^2) \cdot ((X_1 \otimes X_2 \otimes X_3) (T^{(3)}_{\calQ_n}))$ belong to the $0$-th summand. Write $Z$ as sum of simple tensors $e_{i_0} \ootimes e_{i_3}$, and observe that every summand has at least one of the indices $i_1 \vvirg i_3$ equal to $0$ because the product of any three elements of $\frakm$ in $\calQ_n$ equals $0$. Therefore every summand of $Z$ is mapped to $0$ by $ e_0^{(0)*} \otimes e_1^{(1) *} \otimes e_1^{(2) *} \otimes e_1^{(3) *} $. 

Let $Z$ belong to one of the other three summands and without loss of generality assume it is the first one: $Z \in \Hom(\calQ_n^*,\bbC^2) \cdot ((X_0 \otimes X_2 \otimes X_3) (T^{(3)}_{\calQ_n}))$. Similarly, write $Z$ as sum of simple tensors $e_{i_0} \ootimes e_{i_3}$, and observe that either one of the indices $i_2,i_3$ equals $0$, or the summand has the form $e_1 \otimes w \otimes v^{(2)}_j \otimes v^{(3)}_j$ arising from the restriction of the summand $y \otimes 1^* \otimes x_j^* \otimes x_j^*$. In both cases, the summand is mapped to zero by $ e_0^{(0)*} \otimes e_1^{(1) *} \otimes e_1^{(2) *} \otimes e_1^{(3) *} $. The same argument applies to the other two summands in the expression of $\bfT_{T} \calD_n$.

This shows that $\bfT_{T} \calD_n \subseteq \langle e_0^{(0)} \otimes e_1^{(1) *} \otimes e_1^{(2) *} \otimes e_1^{(3) *} \rangle^\perp$, so in particular $\dim \calD_n \leq 15$ for every $n$. On the other hand, a direct computation for $n = 2$ shows that for a generic $T \in \calD_2$, we have $\dim \bfT_{T}\calD_2 = 15$. We deduce that $\calD_n$ is the same hypersurface in $\bbC^2 \ootimes \bbC^2$ for every $n \geq 2$. Moreover, the generic tangent hyperplane, regarded as an element of $\bbP  (\bbC^2 \ootimes \bbC^2)^*$ is a rank one tensor in $\bbP  (\bbC^2 \ootimes \bbC^2)^*$. By genericity, and since $\calD_n$ is closed under the action of $\GL_2 \ttimes \GL_2$, we obtain that the dual variety $\calD_n^\vee$ is the Segre variety $\bbP (\bbC^2)^* \ttimes \bbP (\bbC^2)^*$ in $\bbP(\bbC^2 \ootimes \bbC^2)^* $; the Biduality Theorem \cite[Ch. 1, Thm.~1.1]{GKZ}, guarantees that $\calD_n$ is the dual variety of such Segre variety, which is, by definition, the hypersurface cut out by the hyperdeterminant. 

To conclude that $\uQ(T_{\calQ_n}^{(3)}) \leq 2$, it suffices to prove that $\calD_n$ does not contain all tensors of border rank $3$ in $\bbC^2 \otimes \bbC^2 \otimes \bbC^2 \otimes \bbC^2$. Consider the tensor 
\[
T = e_0^{\otimes 4} + e_1^{\otimes 4} + (e_0+e_1)^{\otimes 3} \otimes (e_0 - e_1).
\]
One can evaluate the hyperdeterminant on $T$ explicitly using Schl\"afli's method, see \cite[Sec. 14.4]{GKZ}, and verify that it does not vanish on $T$. Therefore $T$ is not a degeneration of $T_{\calQ_n}^{(3)}$, showing $\uQ(T^{(3)}_{\calQ_n}) = 2$ for $n \geq 2$.
\end{proof}

We give, for $k=2$, a geometric characterization of the variety of degenerations of $T^{(2)}_{\calQ_n}$ to $\bbC^3 \otimes \bbC^3 \otimes \bbC^3$, similar to the one given in the proof of \cref{border subrank CW k3} in the case $k=3$.
\begin{proposition}
    For every $n \geq 2$, define
    \[
   \calC_n = \{ T \in \bbC^3 \otimes \bbC^3 \otimes \bbC^3: T \text{ is a degeneration of } T^{(2)}_{\calQ_n}\}.
    \]
    If $n \geq 3$, then $\calC_n$ does not depend on $n$ and it is the variety of degenerations of the matrix multiplication tensor $\MaMu_{(2,2,2)}$. In particular $\dim \calC_n = 25$ for all $n \geq 3$.
\end{proposition}
 \begin{proof}
    Given a tensor $T \in \bbC^3 \otimes \bbC^3 \otimes \bbC^3$, we construct three cubic polynomials $\phi^{(0)}_T, \phi^{(1)}_T, \phi^{(2)}_T$ as follows. The $i$-th flattening $T_i : \bbC^3 \to \bbC^3 \otimes \bbC^3$ of $T$ can be regarded as a $3 \times 3$ matrix of linear forms in the $i$-th factor of $T$. The cubic $\phi^{(i)}_T$ is the determinant of such matrix. We prove that for every $T \in \calC_n$, the cubic polynomials $\phi^{(i)}_{T}$ are reducible. Using the results of \cite{BG26} this will provide the desired conclusion. A direct computation in the case $n = 3$ shows $\dim \calC_n \geq 25$ for every $n$.

 Let $T = (X_0 \otimes X_1 \otimes X_2)(T^{(2)}_{\calQ_n})$ be a generic element. After possibly acting with $\GL_3 \times \GL_3 \times \GL_3$, we may assume 
    \begin{align*}
\begin{array}{rlcl}
    X_0 : & y & \mapsto &e^{(0)}_0 \\
          & 1 & \mapsto &e^{(0)}_1 \\
          & x_j & \mapsto & v^{(0)}_j \\ 
    \end{array}
\qquad \begin{array}{rlcl}
    X_i : & 1^* & \mapsto &e^{(i)}_0 \\
          & y^* & \mapsto &e^{(i)}_1 \\
          & x_j^* & \mapsto &v^{(i)}_j
    \end{array} \text{ for $i = 1,2$}
\end{align*}
for generic elements $v^{(i)}_j \in \bbC^3$. Regarding the image of the flattening $T^{(2)}_{\calQ_n}(\calQ_n^*) \subseteq \calQ_n^* \otimes \calQ_n^*$ as a $(n+2) \times (n+2)$ matrix of elements in $\calQ_n$, we have  
\[
T^{(2)}_{\calQ_n}(\calQ_n^*) = \left(\begin{array}{ccccc}
1 & x_1 & \cdots & x_n & y \\ 
x_1 & y & \\
\vdots & & \ddots & \\
x_n  & & & y & \\
y & & & & 
\end{array} \right).
 \]
As a consequence, the image of the flattening of $T$ has the form 
\[
T({\bbC^3}^*) = \left(
\begin{array}{ccc}
z_{00} & z_{01} & z_{02} \\ 
z_{10} & \zeta_{11} e^{(0)}_{0} & \zeta_{12} e^{(0)}_{0} \\
z_{20} & \zeta_{21} e^{(0)}_{0} & \zeta_{22} e^{(0)}_{0} 
\end{array}\right)
\]
for certain vectors $z_{ij} \in \bbC^3$ and certain scalars $\zeta_{ij} \in \bbC$. This shows $\phi^{(0)}_T$ is a multiple of $e^{(0)}_{0}$. In a similar way, we obtain $\phi^{(1)}_T,\phi^{(2)}_T$ are reducible as well. By the characterization of the degenerations of the matrix multiplication tensor $\MaMu_{(2,2,2)}$ in  \cite[Thm. 5.4]{BG26}, we deduce 
\[
\calC_n \subseteq \calM := \{ T \in \bbC^3 \otimes \bbC^3 \otimes \bbC^3 : T \text{ is a degeneration of } \MaMu_{(2,2,2)}\}.
\]
Moreover, by the results of \cite[Sec.~5.1]{BerDLaGes}, we have that $\calM$ is irreducible of dimension $25$, hence $\calC_n = \calM$ for every $n \geq 3$. 
\end{proof}

\begin{proof}[Proof of \cref{thm: subrank higher CW}]
    The statements for $k = 2,3$ follow from \cref{prop: border subrank CW k2} and \cref{border subrank CW k3} respectively. The statement for $ k\geq 4$ follows from \cref{prop: socle prop} since the socle degree of $\calQ_n$ is $2$ for every $n \geq 1$ and $\frakm^2 = \langle y \rangle$ is $1$-dimensional.
\end{proof}

\section{Iterated matrix multiplication}\label{sec: MaMu}

Fix $n$ and consider the algebra $\calA = \Mat_{n\times n}$ of $n \times n$ matrices. Its structure tensors are the (iterated) matrix multiplication tensors. We describe the tensor in the more general setting of iterated multiplication of (possibly rectangular) matrices. In this context, let $\bfn = (n_0 \vvirg n_k)$ be a $(k+1)$-tuple of nonnegative integers and consider the iterated matrix multiplication tensor $\MaMu_\bfn$ defining the multilinear map 
\begin{align*}
 \MaMu_{\bfn} : \Mat_{n_0 \times n_1 } \ttimes \Mat_{n_{k-1} \times n_k} &\to \Mat_{n_0 \times n_{k}} \\
 (X_1 \vvirg X_k) &\mapsto X_1 \cdots X_{k}.
\end{align*}
For $k=2$, border subrank and geometric rank of $\MaMu_{(n_0,n_1,n_2)}$ coincide: the lower bound on border subrank is proved in \cite{Strassen87}, and the upper bound on geometric rank is proved in \cite{KopMosZui:GeomRankSubrankMaMu}; when $n = n_0=n_1=n_2$, this value is $\uQ(\MaMu_{(n,n,n)}) = \GR(\MaMu_{(n,n,n)})=\left\lceil3n^2/4\right\rceil$. 
The main result of this section is the following.
\begin{theorem}\label{thm: main MaMu}
Let $k \geq 3$, $n \geq 2$ be integers. Let $q = \lfloor n/k\rfloor$ and $r = n - qk$. Then
\[
\frac{n^2}{k-1}\leq \uQ(\MaMu_{(n \vvirg n)}) \leq  \frac{1}{2} ( n^2 + nq + r(q+1))
\]
\end{theorem}
The lower bound is obtained by providing an explicit degeneration to a unit tensor of the desired rank. The upper bound is obtained via geometric rank. The value of the upper bound $ \frac{1}{2} ( n^2 + nq + r(q+1))$ is bounded above by the value $\frac{k+1}{2k} n^2 + \frac{k}{8}$ mentioned in \cref{sec: intro}. 

Let $e_{ij}$ denote the element of $\Mat_{n_p \times n_{p+1}}$ having $1$ in the $(i,j)$-th entry and $0$ elsewhere. Let $e_{ij}^*$ denote the elements of the dual basis. Then, as a tensor in $\Mat_{n_0 \times n_k} \otimes \Mat_{n_0 \times n_1}^* \ootimes \Mat_{n_{k-1} \times n_k}^*$, we have
\[
\MaMu_{\bfn} = \sum_{i_0 \vvirg i_k } e_{i_0 i_k} \otimes e_{i_0i_1}^* \ootimes e_{i_{k-1} i_k}^*.
\]
The geometric rank is controlled by the variety 
\[
 \calZ = \{ (X_1 \vvirg X_{k}) : X_1 \cdots X_{k} = 0\} \subseteq  \Mat_{n_0 \times n_1} \ttimes \Mat_{n_{k-1} \times n_k}.
\]
Varieties of this form have been studied in the context of quiver representation theory \cite{AbDf:Deg_Am,AbDfKr:Geo_Am,Ges:Geometry_of_IMM}. Indeed, the $k$-tuple $(X_1 \vvirg X_{k})$ can be regarded as a representation of the equioriented Dynkin quiver 
\[
 \bfA_{k+1} : \begin{tikzpicture}
  \draw (0, 0) -- (1,0);
  \draw (.5+.1 , .1) -- (.5-.1,0);
  \draw (.5+.1 , -.1) -- (.5-.1,0);
\draw[dashed] (1, 0) -- (5,0);
\draw (5, 0) -- (6,0);
  \draw (5.5+.1 , .1) -- (5.5-.1,0);
  \draw (5.5+.1 , -.1) -- (5.5-.1,0);
\draw[fill=white] (0,0) circle(.17); 
\draw[fill=white] (1,0) circle(.17); 
\draw[fill=white] (5,0) circle(.17); 
\draw[fill=white] (6,0) circle(.17); 
\end{tikzpicture}
\]
with dimension vector $\bfn = (n_0 \vvirg n_k)$; the condition $X_1 \cdots X_k = 0$ is equivalent to the fact that the representation is \emph{bound} by the principal ideal generated by the full path. We refer to \cite{DerWey:IntroQuiverRepn} for an extensive treatment of quiver representation theory and to \cite[Sec.~4]{Ges:Geometry_of_IMM} for a brief introduction closely related to the setting of interest.

In our setting, consider vector spaces $U_0 \vvirg U_k$ with $\dim U_j = n_j$ and regard a matrix $X_j \in \Mat_{n_j \times n_{j+1}} \simeq U_{j+1}^* \otimes U_j$ as a linear map $U_{j+1} \to U_j$. Then, a $k$-tuple of linear maps defines a composition  
\[
\xymatrix{
U_0 & U_1 \ar[l]^{X_1} &\ar[l] \cdots & U_{k-1} \ar[l] & U_k \ar[l]^{\ \ \ \ X_{k}}
}
\]
and $(X_1 \vvirg X_{k}) \in \calZ$ if and only if the composition $X_1 \cdots X_k : U_k \to U_0$ is identically $0$.

Representations of the equioriented quiver $\bfA_{k+1}$ decompose into indecomposable representations. By Gabriel's Theorem \cite{Gabriel-quivers}, these are finitely many and in one-to-one correspondence with paths in $\bfA_{k+1}$; a path is uniquely determined by the choice of two nodes $i,j$. In this case the corresponding indecomposable $E_{i,j}$ has dimension vector $(0 \vvirg 0, \underset{i}{1} \vvirg \underset{j}{1} , 0 \vvirg 0)$ and the corresponding maps are nonzero between $1$-dimensional spaces and identically $0$ otherwise. The following immediate result characterizes quiver representations occurring in $\calZ$.

\begin{lemma}
  Let $\bfX = (X_1 \vvirg X_{k})$ be a representation of $\bfA_{k+1}$ with dimension vector $\bfn$. Then $\bfX \in \calZ$ if and only if the indecomposable $E_{0,k}$ does not occur in the decomposition of $\bfX$ into indecomposable representations.
\end{lemma}

Orbits of representations under the action of the subgroup $\calG_{\bfA_{k+1}} = \GL(U_0) \ttimes \GL(U_k)$ are described in \cite{AbDf:Deg_Am,AbDfKr:Geo_Am}. The orbits are uniquely determined by the decomposition into indecomposable. Combinatorially, they can be described as diagrams of boxes with at most $k+1$ columns where each row corresponds to an indecomposable representation. Two diagrams are identified if they are the same up to permutation of rows. There is a combinatorial criterion for the inclusion of an orbit in the closure of another, which can be used to characterize the maximal orbit-closures for which the indecomposable $E_{0,k}$ does not appear: these are the irreducible components of $\calZ$.

We restrict to the case $n_0 = \cdots = n_k = n$. In this case, the irreducible components of $\calZ$ are in one-to-one correspondence with the $k$-tuples $(p_1 \vvirg p_{k})$ such that $p_1 + \cdots + p_k = n$, that is compositions of $n$ into $k$ summands. The value of $p_j$ determines the dimension of the kernel of the map $X_j : U_j \to U_{j-1}$. Write $\calZ_\bfp$ for the irreducible component determined by $\bfp$. Explicitly, $\calZ_\bfp$ is the orbit-closure, under the action of $\calG_{\bfA_{k+1}}$, of the $k$-tuple of matrices $\bfD_\bfp = (D_1 \vvirg D_{k})$, where $D_j$ is a diagonal matrix such that all entries on the diagonal of $D_j$ are one except for the entries ${D_j}_{ii}$ with $p_1 + \cdots + p_{j-1}< i \leq p_1 + \cdots + p_{j}$, which are zero.

By \cite[Prop.~2.7]{AbDf:Deg_Am}, we have the following result.
\begin{proposition}\label{prop: mamu quiver comps}
The irreducible components of $\calZ$ are exactly the varieties 
\[
\calZ_\bfp = \bar{\calG_{\bfA_{k+1}} \cdot \bfD_\bfp}
\]
for every composition $\bfp$ of $n$ into $k$ summands.
\end{proposition}
Given a composition $\bfp = (p_1 \vvirg p_{k})$ define the $k \times k$ \emph{kernel matrix}
\[
 P = \left( \begin{array}{cccccc}
 p_1 & 0 & & &  \\
 \vdots & p_2 & 0 & &  \\
 \vdots & & \ddots & \ddots  &  \\
 & p_{ij} & & p_{k-1} & 0  \\
 & & \cdots & \cdots & p_{k} \\
 \end{array}\right)
\]
where the entries above the diagonal are set to $0$, the diagonal entries are $p_{ii} = p_{i}$ and the entries below the diagonal are defined as 
\[
p_{ij} = p_{i-1,j} + p_{i,j+1}-p_{i-1,j+1} \qquad \text{for } i > j.
\]
An elementary induction argument shows 
\begin{equation}\label{rmk: pij as sums}
p_{ij} = p_{j} +  \cdots + p_{i} \text{ for } i \geq j.
\end{equation}
Formally set $p_{i,0} = n$ for every $i$. Note that if $\bfX \in \calG_{\bfA_{k+1}} \cdot \bfD_\bfp \subseteq \calZ_\bfp$, then 
\[
p_{ij} = \dim \ker (X_j \cdots X_{i}: U_i \to U_{j-1}).
\]
To compute the dimension of $\calZ_\bfp$, we realize $\calZ_\bfp$ as the image of an incidence correspondence. This construction appears in \cite{AbDfKr:Geo_Am} and we recall it here with more details. For $j =0 \vvirg k$, consider the partial flag variety on $U_j$ defined by the sequence of integers $0 = p_{j,j+1}, p_{j,j} \vvirg p_{j,1}, p_{j,0} = n$:
\[
\Flag^{(j)} =  \Flag( p_{j,j} \vvirg p_{j,1}; U_j) = \{ F^{(j)}_\bullet = (0 \subseteq F^{(j)}_j \subseteq \cdots \subseteq F^{(j)}_1 \subseteq U_j ): \dim F^{(j)}_{j'} =  p_{j,j'}\}.
\]
If $j = 0$, $\Flag^{(j)}$ is a single point, identified with the trivial flag $(0 \subseteq U_0)$ in $U_0$. Define an incidence correspondence 
\[
\begin{tikzpicture}
 \node () at (2,2) {$\calI_\bfp = \Bigl\{ (\bfF , \bfX) \in \bigl(\bigtimes_{i=0}^{k}  \Flag^{(i)}\bigr) \times \bigl(\bigtimes_{j= 1}^{k}\Hom( U_{j} , U_{j-1}) \bigr) :  X_j ( F^{(j)}_{j'} ) \subseteq F^{(j-1)}_{j'} \Bigr\} $};
 \node () at (-2,0) {$ \bigtimes_{i=0}^{k}  \Flag^{(i)}  $};
 \node () at (6,0) {$\bigtimes_{j= 1}^{k}\Hom( U_{j} , U_{j-1})  $};
 \draw[->] (1,1.5)--(-1.2,0.5);
 \draw[->] (3,1.5)--(5.2,0.5);
 \node () at (-.8,1) {$\pi_\mathrm{\calF}$};
 \node () at (4.6,1) {$\pi_\mathrm{Hom}$};
 \draw[->] (3,1.5)--(5.2,0.5);
\end{tikzpicture}
\]
Explicitly, if $ (\bfF , \bfX) \in \calI_\bfp$, the map $X_j : U_{j} \to U_{j-1}$ maps the spaces of the flag $F^{(j)}_\bullet$ to the ones of the flag $F^{(j-1)}_\bullet$ as follows.
\[
 \begin{tikzpicture}
\node () at (0,5) {$X_j :$} ;
\node () at (1,5) {$U _{j}$}  ;
\node () at (4,5) {$U _{j-1}$}  ;

\node () at (1,4) {$F^{(j)}_{1}$}  ;
\node () at (1,3) {$F^{(j)}_{2}$}  ;
\node () at (1,2) {$\vdots$}  ;
\node () at (1,1) {$F^{(j)}_{j}$}  ;
\node () at (.6,0) {$0=F^{(j)}_{j+1}$}  ;
\node () at (4,3) {$F^{(j-1)}_{1}$}  ;
\node () at (4,2) {$F^{(j-1)}_{2}$}  ;
\node () at (4,1) {$\vdots$}  ;
\node () at (4.3,0) {$F^{(j-1)}_{j} =0$}  ;

 \draw[->] (1.5,5)--(3.4,5);
 \draw[->] (1.5,3.8)--(3.4,3);
 \draw[->] (1.5,2.8)--(3.4,2);
 \draw[->] (1.5,0.8)--(3.4,0);
 \end{tikzpicture}
\]
Using this construction, we prove the following result:
\begin{theorem}\label{thm: Z for MaMu}
The projection $\pi _\mathrm{Hom}: \calI_{\bfp} \to \bigtimes_{j=1}^{k} \Hom( U_{j} , U_{j-1})$ is birational onto $\calZ_\bfp$. In particular, for every composition $\bfp$ of $n$ into $k$ summands, we have $\dim \calI_{\bfp} = \dim \calZ_\bfp$ and 
\[
 \codim \calZ_\bfp = \sum_{i\leq j} p_ip_j.
\]
\end{theorem}
\begin{proof}
By construction the projection map $\pi_{\mathrm{Hom}}$ surjects onto $\calZ_\bfp$. To show that it is generically one-to-one it suffices to observe that the fiber over $\bfD_\bfp$ is a single point. Indeed, this is the point defined by the $k$-tuple of flags $\bfF \in \bigtimes _{j=0}^{k}  \Flag^{(j)}$ with $F^{(j)}_{j'} = \langle e_1 \vvirg e_{p_{j,j'}}\rangle$ for $j' = j , j-1 \vvirg 1$. Therefore $\dim \calZ_\bfp = \dim \calI_\bfp$.

In order to compute $\dim \calI_\bfp$, note that $\pi_\calF : \calI_\bfp \to \bigtimes_{j=0}^{k} \Flag^{(j)}$ is a vector bundle over the product $\bigtimes_{j=0}^{k} \Flag^{(j)}$. The fiber over $\bfF$ is the product of linear spaces $\calL_\bfF =  L^{(1)}_{\bfF} \ttimes  L^{(k)}_{\bfF}$ where 
\[
 L^{(j)}_{\bfF} = \{ X_j : U_{j} \to U_{j-1} : X_j ( F^{(j)}_{j'} ) \subseteq F^{(j-1)}_{j'} \text{ for } j' = 1 \vvirg j\}.
\]
    We record two dimension formulas that we will use to compute the dimension of $\calZ_\bfp$; we refer to \cite[Sec.~1.2]{BrionFlags} and \cite{AbDfKr:Geo_Am}. For fixed integers $0 = m_0 \leq \cdots \leq m_s = t$ the dimension of the flag variety $\Flag (m_0 \vvirg m_s, \bbC^t)$ in $\bbC^t$ is
\begin{equation}\label{rmk: dim flag}
 \dim \Flag(m_0 \vvirg m_s,\bbC^t) = \sum_{i=1}^{s-1} m_i (m_{i+1}-m_i).
\end{equation}
Moreover, given two flags 
\[ 
F_\bullet ^{1} \in \Flag(m_0^{(1)} \vvirg m_s^{(1)} ; \bbC^{t^{(1)}}) \quad \text{and} \quad F_\bullet ^{2} \in \Flag(m_0^{(2)} \vvirg m_{s-1}^{(2)} ; \bbC^{t^{(2)}}),
\]
with $m_0^{(1)} = m_0^{(2)} =  0$ and $m_{s}^{(1)} = t^{(1)}$, $m_{s-1}^{(2)} = t^{(2)}$, we have
\begin{equation}\label{rmk: dim maps preserving flag}
 \dim \{ X : \bbC^{t^{(1)}} \to \bbC^{t^{(2)}}  : X ( F_j^{1} ) \subseteq F_{j-1}^2\} = \sum_{i=1}^{s-1} (m^{(1)}_{i+1}-m_{i}^{(1)}) m^{(2)}_{i}.
\end{equation}

By the fiber dimension theorem, see, e.g. \cite[Cor. 11.13]{Harris:AlgGeo}, we have
\begin{align*}
\dim \calZ_\bfp = \dim \calI_\bfp &= \sum _0^{k} \dim \Flag^{(j)} + \dim \pi_\calF^{-1}(F^{(1)}_\bullet \vvirg F^{(k)}_\bullet).  
\end{align*}
where $\pi_\calF^{-1}(F^{(1)}_\bullet \vvirg F^{(k)}_\bullet)$ is the preimage of a generic element $(F^{(1)}_\bullet \vvirg F^{(k)}_\bullet)$ of $\bigtimes_{i=0}^{k}  \Flag^{(i)} $. From \eqref{rmk: dim flag}, we obtain the first summand
\[
 \sum _{j=0}^{k} \dim \Flag^{(j)} = \sum_{j=0}^k \Flag(\underbrace{p_{j,j+1}}_{=0},\underbrace{p_{jj}}_{=p_j} \vvirg p_{j1} , \underbrace{p_{j0}}_{=n}, \bbC^n) = \sum_{j=0}^k \sum_{i=1}^{j} p_{j,i}(p_{j,i-1}-p_{j,i}).
\]
For $j=0$ the inner summation is empty. Therefore, the range of the outer summation is $j=1 \vvirg k$ and shifting the indices we obtain
\begin{align*}
 \sum _{j=0}^{k} \dim \Flag^{(j)} &= \sum_{j=1}^{k}  \sum_{i=1}^{j} p_{j,i}(p_{j,i-1}-p_{j,i}) = \sum_{j=0}^{k-1}  \sum_{i=1}^{j+1} p_{j+1,i}(p_{j+1,i-1}-p_{j+1,i}) .
\end{align*}
Separate the term $i=1$ and shift the indices in the inner summation to obtain 
\begin{align*}
  \sum _{j=0}^{k} \dim \Flag^{(j)} = &\sum_{j=0}^{k-1} p_{j+1,1}(p_{j+1,0}-p_{j+1,1}) + \sum_{j=0}^{k-1} \sum_{i=2}^{j+1} p_{j+1,i}(p_{j+1,i-1}-p_{j+1,i}) = \\ 
  &\sum_{j=0}^{k-1} p_{j+1,1}(p_{j+1,0}-p_{j+1,1}) + \sum_{j=0}^{k-1} \sum_{i=1}^{j} p_{j+1,i+1}(p_{j+1,i}-p_{j+1,i+1}).
\end{align*}
Applying \eqref{rmk: pij as sums} and shifting the indices again, we obtain
 \begin{align*}
  \sum _{j=0}^{k} \dim \Flag^{(j)} & = \sum_{j=0}^{k-1} (p_1 + \cdots + p_{j+1})(n - (p_1 + \cdots + p_{j+1}) ) +  \sum_{j=0}^{k-1} \sum_{i=1}^{j} (p_{i+1} + \cdots + p_{j+1})p_i \\ 
 &=\sum_{j=1}^{k} n(p_1 + \cdots + p_j) - \sum_{j=1}^{k} (p_1 + \cdots + p_j)^2 + \sum_{j=1}^{k} \sum_{i=1}^{j-1} (p_{i+1} + \cdots + p_{j})p_i. 
\end{align*}
Now, applying \eqref{rmk: dim maps preserving flag} for the linear spaces $L^{(j)}_{\bfF}$, we compute $\dim \pi_\calF^{-1}(\bfF)$. We have 
\[
 \dim \pi_\calF^{-1}(\bfF) = \sum _{j=1}^{k} \dim L^{(j)}_{\bfF} =  \sum_{j=1}^{k} \sum_{i=0}^{j-1}  p_{j-1,i}(p_{j,i} - p_{j,i+1}).
\]
As before, separating the term $i = 0$ of the inner summation, observe that the remaining sum is empty for $j=1$ and obtain 
\begin{align*}
  \dim \pi_\calF^{-1}(\bfF) &= \sum_{j=1}^{k}  p_{j-1,0}(p_{j,0} - p_{j,1}) + \sum_{j=1}^{k} \sum_{i=1}^{j-1}  p_{j-1,i}(p_{j,i} - p_{j,i+1})= \\
			    &= \sum_{j=1}^{k} n(n - (p_{1} + \cdots + p_{j})) + \sum_{j=2}^{k} \sum_{i=1}^{j-1}  p_{j-1,i}(p_{j,i} - p_{j,i+1})=\\
			    &= k\cdot n^2 -  \sum_{j=1}^{k} n(p_{1} + \cdots + p_{j}) + \sum_{j=1}^{k-1} \sum_{i=1}^{j}  (p_i + \cdots + p_{j}) p_i .
\end{align*}
Adding the contributions together and passing to codimensions, we have
\begin{align*}
\codim \calZ_\bfp = &\sum_{j=1}^k \dim ( \Hom( U_j, U_{j-1}) ) - \left\{ \sum _{j=0}^{k} \dim \Flag^{(j)} + \dim \pi_\calF^{-1}(\bfF)\right\} = \\
&k \cdot n^2 - \\
&\left\{ \left[ \sum_{j=1}^{k} n(p_1 + \cdots + p_j) - \sum_{j=1}^{k} (p_1 + \cdots + p_j)^2 + \sum_{j=1}^{k} \sum_{i=1}^{j-1} (p_{i+1} + \cdots + p_{j})p_i \right]\right. \\
		  &+ \left.\left[ k\cdot n^2 -  \sum_{j=1}^{k} n(p_{1} + \cdots + p_{j}) + \sum_{j=1}^{k-1} \sum_{i=1}^{j}  (p_i + \cdots + p_{j}) p_i \right]\right\}.	  
\end{align*}
After the immediate cancellations of the terms $k \cdot n^2$ and $\sum_{j=1}^{k} n(p_1 + \cdots + p_j)$, we obtain
\begin{equation}\label{eqn: codim Z start}
 \codim \calZ_\bfp = \sum_{j=1}^{k} (p_1 + \cdots + p_j)^2 - \sum_{j=1}^{k} \sum_{i=1}^{j-1} (p_{i+1} + \cdots + p_{j})p_i - \sum_{j=1}^{k-1} \sum_{i=1}^{j}  (p_i + \cdots + p_{j}) p_i .
\end{equation}
Observe 
\begin{align*}
 \sum_{j=1}^{k} (p_1 + \cdots + p_j)^2 &=  \sum_{j=1}^{k-1} (p_1 + \cdots + p_j)^2 + (p_1 + \cdots + p_{k})^2 = \\
					 &=   \sum_{j=1}^{k-1} (p_1 + \cdots + p_j)^2 + \sum_{i=1}^{k-1} p_i (p_{i+1} + \cdots + p_{k}) + h_2(\bfp),
\end{align*}
where $h_2(\bfp) = \sum_{i \leq j} p_i p_j$ is the second complete symmetric function evaluated at $\bfp$, that is value of $\codim \calZ_\bfp$ in the statement of the theorem.

The second summand above cancels the $j=k$ term in the second term in \eqref{eqn: codim Z start}. To conclude we show that the sum of the remaining terms equals $\sum_{j=1}^{k-1} (p_1 + \cdots + p_j)^2$, so that they cancel and only $h_2(\bfp)$ is left. Indeed, for every $j$, we have 
\begin{align*}
&\sum_{i=1}^{j-1} (p_{i+1} + \cdots + p_{j})p_i + \sum_{i=1}^{j}  (p_i + \cdots + p_{j}) p_i = \\ 
&\sum_{1 \leq i < i' \leq j} p_{i} p_{i'} + \sum_{1 \leq i < i' \leq j} p_{i'}p_i +  \sum_{i=1}^j p_i^2  = (p_1 + \cdots + p_j)^2.
\end{align*}
Adding the contribution for every $j = 1 \vvirg k-1$, we conclude.
\end{proof}
Using \cref{thm: Z for MaMu}, we compute the geometric rank of the iterated matrix multiplication tensor.
\begin{corollary}\label{corol: GR MaMu}
Let $k \geq 3$ and $n \geq 2$ be integers. Let $q = \lfloor n/k \rfloor$ and $r = n - kq$. Let $\bfp = ((q+1)^r, q^{k-r})$ be the composition of $n$ consisting of $r$ copies of $(q+1)$ and $k-r$ copies of $q$. Then 
\[
\GR(\MaMu_{(n \vvirg n)}) = \codim \calZ_\bfp = h_2(\bfp) =   \frac{1}{2} ( n^2 + nq + r(q+1)).
\]
In particular, if $n$ is divisible by $k$, then $\bfp = (q \vvirg q)$ and 
\[
    \GR(\MaMu_{(n \vvirg n)}) = h_2(q \vvirg q) =  \frac{k+1}{2k} n^2.
\]
\end{corollary}
\begin{proof}
    The complete symmetric function $h_2$ is Schur-convex, see, e.g., \cite{GuanSchurConvex}. Therefore, $h_2(\bfp)$ is minimized when the parts of $\bfp$ are as evenly distributed as possible that is when $\bfp$ is the composition described in the statement. 
\end{proof}
 \cref{corol: GR MaMu} provides the upper bound on the border subrank of $\MaMu_{(n \vvirg n)}$ in \cref{thm: main MaMu}. For $k = 2$, this recovers the upper bound of \cite{KopMosZui:GeomRankSubrankMaMu}, which is tight. In \cite{KopMosZui:GeomRankSubrankMaMu}, a tight upper bound was given also in the rectangular case, when $n_0,n_1,n_2$ are not necessarily equal. An analog of \cref{prop: mamu quiver comps} holds when $n_0 \vvirg n_k$ are not all equal as well, but the dimension analysis in the proof of \cref{thm: Z for MaMu} is much more complicated. 

A lower bound on the border subrank of $\MaMu_{(n \vvirg n)}$ can be obtained from \cite[Thm.~3]{VrChr}, which provides a lower bound for border subrank of any graph tensor, in the sense of \cite{ChrVraZui:AsyRankGraph}. Specifically, for a tensor associated to a certain graph, the lower bound relies on the existence of general-position orthogonal representations of the associated \emph{line graph}. In the case of iterated matrix multiplication, the graph defining the tensor is the cycle $C_{k+1}$ with $k+1$ vertices, which coincides with its own line graph. In this case, \cite[Thm.~3]{VrChr}, using the results of \cite{LSS89}, gives a lower bound $\uQ(\MaMu_{(n \vvirg n)}) \geq \Omega(n^2)$ for every fixed $k$. Instead of relying on \cite{LSS89}, we give an explicit construction of a general-position orthogonal representation, which yields a lower bound with explicit dependence on $n$ and $k$. For $k=3$, our lower bound recovers the lower bound $\uQ(\MaMu_{(n ,n,n, n)}) \geq \lceil n^2/2\rceil$ from \cite[Sec.~5.2]{ChrLucVraWer}.

\begin{proposition}
    Let $k\geq 3$, $n\geq 2$. Then 
    \[
    \uQ(\MaMu_{(n \vvirg n)})\geq \frac{n^2}{k-1}.
    \]
\end{proposition}

\begin{proof}
Consider the general-position orthogonal representation of $C_{k+1}$ in $\mathbb Z^{k-1}$ given by
\begin{align*}
 & c_0= e_1, \\
 &c_i = e_{i}+e_{i+1} \quad \text{ for } 1\le i\le k-2, \\
 &c_{k-1} = e_{k-1}, \\
 &c_{k}  =e_1-e_2+e_3-\cdots+(-1)^{k}e_{k-1} .
\end{align*}
These vectors satisfy $\langle c_i,c_j\rangle=0$ whenever $i$ and $j$ are nonadjacent in the cycle, so \cite[Thm.~3]{VrChr} applies. For every $h\in\mathbb Z^{k-1}$, there is a curve $g(\eps)\in\GL(\Mat_{n\times n})\times\GL(\Mat_{n\times n}^*)^{\times k}$ such that 
\[
g(\eps)\MaMu_{(n,\dots,n)}=\sum_{i_0 \vvirg i_k \in\{0,\dots,n-1\}}\eps^{|c_0i_0+\cdots+c_{k}i_{k}-h|^2} e_{i_0 i_k} \otimes e_{i_0i_1}^* \ootimes e_{i_{k-1} i_k}^*.
\]
Then the limit $\lim_{\eps\to0}g(\eps)\MaMu_{(n,\dots,n)}$ is isomorphic to the unit tensor with size equal to 
the number of solutions of the linear system,
\[
c_0i_0+\cdots+c_{k}i_{k}=h\text{ with }i_0,\dots,i_k\in\{0,\dots,n-1\}, 
\] 
of $k-1$ equations in $i_0,\dots,i_k\in\{0,\dots,n-1\}$.

Consider $h=h_1e_1+\cdots+h_{k-1}e_{k-1}\in\mathbb Z^{k-1}$ with $h_\ell= (n-1)-(-1)^\ell q$ where $q = \lfloor \frac{n-1}{k-1} \rfloor$. The system of linear equations reduces to
\[
\begin{array}{cc}
    i_0+\ell i_k=\begin{cases}
        n-1+\ell q-i_\ell &\text{ if }\ell\text{ is odd,} \\
        \ell q +i_\ell&\text{ if }\ell\text{ is even,}
    \end{cases} & \text{ with }i_0,\dots,i_k\in\{0,\dots,n-1\}.
\end{array}
\]
Therefore, every choice of $i_0,i_k$ uniquely determines a solution, as long as, for every $\ell = 0 \vvirg k-1$, the following inequality holds:
\[
\begin{array}{ccc}
    \ell q \leq i_0+\ell i_k\leq n-1+\ell q & \text{ with } i_0,i_k\in\{0,\dots,n-1\}.
\end{array}
\]
The inequalities for $\ell = 0$ and for $\ell = k-1$ together imply all others. The number of solutions of the linear system is the number of tuples $(i_0,i_{k})\in\{0,\dots,n-1\}^{\times 2}$ satisfying 
\[
(k-1) q \leq i_0+(k-1) i_k\leq n-1+(k-1)q.
\]
This number is 
\[
M(n,k)=\sum_{i_0=0}^{n-1}\left(\left\lfloor\frac{n-1-i_0}{k-1}\right\rfloor+\left\lfloor\frac{i_0}{k-1}\right\rfloor+1\right)=n+2\sum_{i=0}^{n-1}\left\lfloor\frac{i}{k-1}\right\rfloor=n+q(n-k+r+2)
\]
where $r=(n-1)-q(k-1)$. The resulting degeneration selects exactly a family of $M(n,k)$ surviving summands, yielding 
\[
\uQ(\MaMu_{(n \vvirg n)}) \geq M(n,k) \geq \frac{n^2}{k-1} .
\]
This concludes the proof.
\end{proof}
For iterated matrix multiplication, unlike the case with three tensor factors, border subrank and geometric rank can be different. In particular, the upper bound of \cref{thm: main MaMu} is not necessarily sharp, as shown in the next example.  
\begin{example}\label{ex: MaMu 2222}
Consider $n = 2$, $k=3$, so that $\MaMu_{(2,2,2,2)}$ is the structure tensor of the $3$-fold product in $\Mat_{2 \times 2}$. The geometric rank computed in \cref{corol: GR MaMu} is $\GR(\MaMu_{(2,2,2,2)}) = 3$. We prove that $\rmQ(\MaMu_{(2,2,2,2)}) = \uQ(\MaMu_{(2,2,2,2)}) = 2$. The lower bound is clear: the restriction to $2 \times 2$ diagonal matrices shows $\rmQ(\MaMu_{(2,2,2,2)}) \geq 2$. If $\uQ(\MaMu_{(2,2,2,2)}) = 3$, then the variety
\[
\calM := \{ T \in \bbC^2 \otimes \bbC^2 \otimes \bbC^2 \otimes \bbC^2 : T \text{ is a degeneration of } \MaMu(2,2,2,2) \}
\]
of degenerations of $\MaMu_{(2,2,2,2)}$ would contain all tensors of border rank at most $3$. In \cite[Thm.~5.3]{BerDLaGes}, it was shown that $\calM$ is the hypersurface of degree $6$ defined by the invariant $F_6$ described in \cref{sec: invariants 2222}. Such an equation does not vanish on the set of tensors of border rank at most $3$; for instance, the tensor 
\[
T = e_0^{\otimes 4} + e_1^{\otimes 4} + (e_0 + e_1)^{\otimes 4} 
\]
is not in $\calM$, but it is immediate that $\rmR(T) = \uR(T) = 3$.
\end{example}

\section{Triangular matrix multiplication}\label{sec: TMaMu}

Let $\calT_n$ be the algebra of $n \times n$ upper triangular matrices. Note $\dim \calT_n = n(n+1)/2$. Consider the structure tensor $T^{(k)}_{\calT_n}$ of the $k$-fold multiplication in $\calT_n$. In this section, we compute the border subrank and the geometric rank of $T^{(k)}_{\calT_n}$, which turn out to be equal. We provide a lower bound on the subrank of $T^{(k)}_{\calT_n}$, which is equal to the border subrank when $k\geq n/2$.

Let $e_{ij}$ for $i \leq j$ be the basis of $\calT_n$ where $e_{ij}$ is the matrix with $1$ on the $(i,j)$-entry and zero elsewhere. Write $e_{ij}^*$ for the dual basis of $\calT_n^*$. The structure tensor of $k$-fold multiplication in~$\calT_n$ is
\[
T_{\calT_n}^{(k)} = \sum_{1 \leq i_0 \leq \cdots \leq i_{k} \leq n} e_{i_0i_{k}} \otimes e^*_{i_0i_1} \ootimes e^*_{i_{k-1}i_{k}}.
\]
In this section, we prove the following result:
\begin{theorem}\label{thm: subrank triangular}
Let $n,k \geq 1$ be integers, and let $q=\left\lfloor\frac{n}{k}\right\rfloor$. Then 
\[
\uQ(T^{(k)}_{\calT_n}) = \GR(T^{(k)}_{\calT_n}) = \frac{(q+1)(2n-qk)}{2}.
\]
\end{theorem}
The upper bound in \cref{thm: subrank triangular} is obtained by giving an upper bound to geometric rank in \cref{prop: triangular GR upper bound}. The lower bound is obtained by an explicit degeneration to a unit tensor of the desired size. The lower bound in the case $k=2$ was already computed in \cite{Strassen87}. In \cref{prop: triangular Q lower bound}, we provide a lower bound on the subrank of $ T^{(k)}_{\calT_n}$, which matches the value of the border subrank when $k \geq n/2$. 

The following result determines an upper bound on the geometric rank of $T^{(k)}_{\calT_n}$, which gives the upper bound in \cref{thm: subrank triangular}. 
\begin{proposition}\label{prop: triangular GR upper bound}
Let $n,k \geq 1$ be integers, and let $q=\left\lfloor\frac{n}{k}\right\rfloor$. Then 
\[
\GR(T^{(k)}_{\calT_n}) \leq \frac{(q+1)(2n-qk)}{2}.
\]
\end{proposition}
\begin{proof}
Write $n = qk+r$ and consider the blocking of a matrix of size $n \times n$ into blocks of sizes $(\underbrace{(q+1)\vvirg (q+1)}_{r \text{ times}} , \underbrace{q \vvirg q}_{k-r \text{ times}} )$. A triangular matrix then is partitioned into blocks, with $k$ triangular blocks on the diagonal of size either $(q+1) \times (q+1)$ or $q \times q$. 

Consider the multilinear subspace $L_1 \ttimes L_k \subseteq \calT_n^{\times k}$ where $L_j$ has zeroes on the $j$-th diagonal block; in particular $\codim L_j = \dim \calT_{q}$ or $\codim L_j = \dim \calT_{q+1}$ depending on the size of the block. As a multilinear map $T^{(k)}_{\calT_n}$ is identically zero on $L_1 \ttimes L_k$ hence
\begin{align*}
\GR(T^{(k)}_{\calT_n}) &\leq \codim (L_1 \ttimes L_k) \\
&= r \cdot \dim \calT_{q+1} + (k-r) \cdot \dim \calT_q \\
&= r \cdot (q+1)(q+2)/2 + (k-r) \cdot q(q+1)/2 =  (q+1)(2n - qk)/2.\qedhere
\end{align*}
\end{proof}

The lower bound in \cref{thm: subrank triangular} is obtained by providing an explicit degeneration of $T^{(k)}_{\calT_n}$ to a unit tensor of size $(q+1)(2n-qk)/2$.

\begin{proof}[{Proof of \cref{thm: subrank triangular}}]
We are going to show that $T^{(k)}_{\calT_n}$ degenerates to the tensor 
\[
\sum_{\substack{ d\geq0 \\ a+kd\leq n}} e_{a,a+kd}\otimes e_{a,a+d}^*\otimes e_{a+d,a+2d}^*\otimes\cdots\otimes e_{a+(k-1)d,a+kd}^*,
\]
and we will show this is isomorphic to the unit tensor of size $(q+1)(2n-qk)/2$. Define maps $X_0 \vvirg X_k$, depending on a parameter $\eps$, as follows:
\[
   \begin{array}{rcl} 
    X_0 : \calT_n &\to & \calT_n \\
      e_{a,b}& \mapsto &\eps^{-(b-a)^2}e_{a,b} 
      \end{array}, \qquad
\begin{array}{rcl}
X_i : \calT_n^* &\to & \calT_n^* \\
   e_{a,b}^* &\mapsto &\eps^{k(b-a)^2}e_{a,b}^* 
\end{array} \text{ for } i = 1 \vvirg k.
\]
Then 
\begin{align*}
    (X_0\otimes&\cdots\otimes X_k)(T^{(k)}_{\calT_n}) \\[5pt]
    &= \sum_{1 \leq i_0 \leq \cdots \leq i_{k} \leq n} \eps^{-(i_k-i_0)^2+k[(i_1-i_0)^2+\cdots+(i_k-i_{k-1})^2]}\cdot e_{i_0i_{k}} \otimes e^*_{i_0i_1} \ootimes e^*_{i_{k-1}i_{k}}.
\end{align*}
We are going to prove that the exponents of $\eps$ in the summation above are strictly positive unless $i_1-i_0=\cdots=i_k-i_{k-1}$, in which case they are zero. 
For $1 \leq i_0 \leq \cdots \leq i_{k} \leq n$, using Cauchy–Schwarz inequality, we obtain 
\[
k\left[(i_1-i_0)^2+\cdots+(i_k-i_{k-1})^2\right]\geq \left[(i_1-i_0)+\cdots+(i_k-i_{k-1})\right]^2 = (i_k-i_0)^2.
\]
This shows that the exponents of $\eps$ in the degeneration are always nonnegative. Moreover, they are zero if and only if $i_1-i_0=\cdots=i_k-i_{k-1}$. Let $a=i_0$ and $d=i_1-i_0=\cdots=i_k-i_{k-1}$. Passing to the limit, we obtain
\[
\lim_{\eps\to0}(X_0\otimes\cdots\otimes X_k)(T^{(k)}_{\calT_n})=\sum_{\substack{ d\geq0 \\[3pt] 1\leq a\leq n-kd}} e_{a,a+kd}\otimes e_{a,a+d}^*\otimes e_{a+d,a+2d}^*\otimes\cdots\otimes e_{a+(k-1)d,a+kd}^*,
\]
which is isomorphic to the unit tensor of size equal to the cardinality of the set of pairs
\[
S=\Bigl\{(a,d):d=0,1,\dots,\lfloor (n-1)/k \rfloor\text{ and }a=1,2,\dots,n-kd\Bigr\}.
\]
The cardinality of this set is 
\[
|S| = \sum_{d=0}^{\left\lfloor\frac{n-1}{k}\right\rfloor}(n-kd) = \frac{(\left\lfloor\frac{n-1}{k}\right\rfloor+1)(2n-\left\lfloor\frac{n-1}{k}\right\rfloor k)}{2}.
\]
Write $n=qk+r$ where $q = \left\lfloor\frac{n}{k}\right\rfloor$ and $r=0,\dots,k-1$. 
If $r\neq0$, then $\left\lfloor\frac{n-1}{k}\right\rfloor=q$ and $|S|=\frac{(q+1)(2n-qk)}{2}$.
If $r=0$, then $n=qk$, $\left\lfloor\frac{n-1}{k}\right\rfloor=q-1$, and 
\[
|S|=\frac{q(2n-qk+k)}{2}=\frac{qk(q+1)}{2}=\frac{(q+1)(2n-qk)}{2}.\qedhere
\]
\end{proof}

We provide a lower bound on the subrank of $T^{(k)}_{\calT_n}$ built on the combinatorics of $k$-average free sets. We say that a subset $D \subseteq \bbZ$ is \emph{$k$-average free} if the only choice of $k+1$ elements $x_1 \vvirg x_k,y \in D$ satisfying $x_1+\cdots+x_k = ky$ is $x_1=\cdots=x_k=y$.

\begin{proposition}\label{prop: triangular Q lower bound}
Let $n,k \geq 1$ be integers and let $q = \lfloor (n-1)/k \rfloor$. Let $D\subseteq\{0 \vvirg q\}$ be a $k$-average free set, then 
\[
\rmQ(T^{(k)}_{\calT_n}) \geq \sum_{d\in D}(n-kd).
\]
In particular, if $n/2 \leq k \leq n-1$, then $\rmQ(T^{(k)}_{\calT_n}) = \uQ(T^{(k)}_{\calT_n}) = 2n-k$ and if $k \geq n$ then $\rmQ(T^{(k)}_{\calT_n}) = \uQ(T^{(k)}_{\calT_n}) = n$.
\end{proposition}
\begin{proof}
 For $p = 1 \vvirg k$, consider the restriction on the $p$-th factor $\calT_n^*$ given by 
\[
\begin{array}{rll}
    X_p : \calT_n^* &\to \calT_n^* &\\ 
    e_{a+(p-1)d,a+pd}^* &\mapsto e_{a+(p-1)d,a+pd}^*  & \text{for }d\in D\text{ and }a = 1 \vvirg n-kd \\
    e_{ij}^* &\mapsto 0 & \text{otherwise}.
\end{array}
\]
On the $0$-th factor $\calT_n$, consider the restriction
\[
\begin{array}{rll}
    X_{0} : \calT_n &\to \calT_n &\\
    e_{a,a+kd} &\mapsto e_{a,a+kd}  & \text{for }d\in D\text{ and }a = 1 \vvirg n-kd \\
    e_{ij} &\mapsto 0 & \text{otherwise}.
\end{array}
\]
This restriction zeroes out all terms in $T^{(k)}_{\calT_n}$ except for the sum of the terms 
\[
 e_{i_0i_{k}} \otimes e^*_{i_0i_1} \ootimes e^*_{i_{k-1}i_{k}}
\]
with $1 \leq i_0 \leq \cdots \leq i_{k} \leq n$, such that $i_p - i_{p-1} = d_p$ for some $d_p\in D$ and $i_k=i_0+kd$ for some $d\in D$. We obtain $i_0+kd = i_k = i_0+d_1+\cdots+d_k$, therefore $d_1 + \cdots + d_k = k d$ with $d_1 \vvirg d_k,d \in D$. Since $D$ is $k$-average free, we deduce $d_1=\cdots=d_k=d$, so that the restriction is 
\[
(X_0\otimes\cdots\otimes X_k)(T^{(k)}_{\calT_n})=\sum_{\substack{ d\in D \\[3pt] 1\leq a\leq n-kd}} e_{a,a+kd}\otimes e_{a,a+d}^*\otimes e_{a+d,a+2d}^*\otimes\cdots\otimes e_{a+(k-1)d,a+kd}^*,
\]
which is isomorphic to the unit tensor $\bfu_{k+1}(\sum_{d\in D}(n-kd))$.

If $k \geq n/2$, then $q = 1$ and $D = \{0,1\}$ is clearly a $k$-average free subset of $\{0,1\}$. In this case, we obtain $\rmQ(T^{(k)}_{\calT_n}) \geq 2n-k$ which matches the value of the border subrank in \cref{thm: subrank triangular}.

If $k \geq n$, then $q=0$ and the construction above yields the restriction to diagonal matrices. This gives $\rmQ(T^{(k)}_{\calT_n}) \geq n$ which matches the value of the border subrank in \cref{thm: subrank triangular}.
\end{proof}

If $n\geq 2k$, then $q=\lfloor\frac{n-1}{k}\rfloor\geq 2$. In this case, a $k$-average free set $D$ must be a proper subset of $\{0,1,\dots,q\}$ because $0+(k-2)\cdot1+2 = k\cdot 1$. In particular, the lower bound of \cref{prop: triangular Q lower bound} is always strictly smaller than the value of the border subrank computed in \cref{thm: subrank triangular}.

\section{Special linear Lie algebra}\label{sec: lie sln}

In this section, we discuss structure tensors of Lie algebras, that is the tensor corresponding to the bilinear map defined by the Lie bracket. For a Lie algebra $\frakg$, let 
\[   
T^{(2)}_\frakg : \frakg \times \frakg \to \frakg, \qquad (X,Y) \mapsto [X,Y]
\]
be the structure tensor of the Lie bracket and $T^{(k)}_\frakg :(X_1 \vvirg X_k) \mapsto [\cdots[X_1,X_2],X_3]\cdots,X_k]$ be the structure tensor of the $k$-fold Lie bracket. We will focus on the case $\frakg = \fraksl_n$, the Lie algebra of traceless $n \times n$ matrices with the commutator bracket. First, we provide some basic observations about general Lie algebras. We refer to \cite{Hall} for background material on this topic. 

For a given Lie algebra $\frakg$, let $\frakz \subseteq \frakg$ be its center, that is the subspace $\frakz = \{ Z \in \frakg : [Z,\frakg] = 0\}$ and let $\frakg' = [\frakg,\frakg]$ be the derived subalgebra. If $\frakz \neq 0$, the structure tensor $T_\frakg^{(2)} \in \frakg \otimes \frakg^* \otimes \frakg^*$ is not concise, because it belongs to the subspace $\frakg' \otimes \frakz^\perp \otimes \frakz^\perp$. More generally, define $\frakg^{(p)} = [\frakg^{(p-1)},\frakg]$ and $\frakz^{(p)} = \{ Z \in \frakg : [Z,\frakg^{(p-1)}] = 0\}$; in particular $\frakg' = \frakg^{(1)}$ and $\frakz = \frakz^{(1)}$. Then 
\[
T^{(k)}_\frakg \in \frakg^{(k-1)} \otimes {\frakz^{(1)}}^\perp \otimes {\frakz^{(1)}}^\perp \otimes {\frakz^{(2)}}^\perp \ootimes {\frakz^{(k-1)}}^\perp;
\]
for example, if $\frakg$ is nilpotent, in the sense that $\frakg^{(k_0)} = 0$ for some $k_0$, then $T^{(k)}_\frakg = 0$ if $k \geq k_0$. This highlights the importance of the identity elements in propagation results such as \cref{thm: degeneration between higher order algebras}.

If $\frakg$ is reductive, then $\frakg = \fraks \oplus \frakz$ where $\fraks$ is semisimple. In this case $T^{(k)}_\frakg = T^{(k)}_\fraks \in \fraks \otimes \fraks^* \ootimes \fraks^*$. We provide a general result on geometric rank in this case, restricting to the semisimple setting.
\begin{proposition}\label{prop: GR semisimple}
Let $\frakg$ be a semisimple Lie algebra of dimension $n$ and rank $r$. Then 
\[
\GR(T^{(2)}_\frakg) = n - r.
\]
\end{proposition}
\begin{proof}
The geometric rank of $T^{(2)}_\frakg$ is the codimension of the variety 
\[
\calZ_2(\frakg) = \{ (X,Y) \in \frakg \times \frakg : [X,Y] = 0\} \subseteq \frakg \times \frakg.
\]
A consequence of \cite[Thm. A]{Richardson} is that the variety $\calZ_2(\frakg)$ is irreducible. Consider the projection $\pi: \calZ_2(\frakg) \to \frakg$ on the first factor and let $S \subseteq \frakg$ be the open subset of regular semisimple elements. Then $\pi : \pi^{-1}(S) \to S$ is a vector bundle: the fiber over $X \in S$ is the centralizer $C_\frakg(X)$, which is a Cartan subalgebra of $\frakg$ so it has dimension $r$. The irreducibility of $\calZ_2(\frakg)$ guarantees that $\pi^{-1}(S)$ is Zariski dense in $\calZ_2(\frakg)$; so we obtain
\[
\dim \calZ_2(\frakg)  = \dim S + \dim C_\frakg(X) = n+r.
\]
We conclude $\GR(T^{(2)}_\frakg) = 2n - \dim \calZ_2(\frakg) = n-r$.
\end{proof}
In the following, we focus on the case $\frakg = \fraksl_n$; note $\dim \fraksl_n = n^2-1$. The main result of this section is the following.
\begin{theorem}\label{thm: sl lie algebra}
    Let $n\geq 2$ and $k\geq 2$. Then 
    \begin{itemize}
    \item if $k = 2$ then 
    \[
    n \leq \rmQ(T^{(2)}_{\mathfrak{sl}_n}) \leq \uQ(T^{(2)}_{\mathfrak{sl}_n}) \leq n^2-n;
    \]
    \item if $n = 2$ then $\uQ(T^{(k)}_{\mathfrak{sl}_2}) = 2$ for every $k$. 
    \end{itemize}
\end{theorem}

The upper bounds of \cref{thm: sl lie algebra} follow from \cref{prop: GR semisimple} since $\fraksl_n$ has dimension $n^2-1$ and rank $n-1$. \cref{prop: Q T2sln} and \cref{prop: uQ Tksl2} prove the lower bound in the first part and the second part respectively. In the following, let $e_{i,j}$  be the matrix with $1$ on the $(i,j)$-entry and zero elsewhere, and let $h_i=e_{i,i}-e_{i+1,i+1}$. Then $e_{i,j}$ for $i \neq j$ and $h_1,\dots,h_{n-1}$ define a basis of $\fraksl_n$. Let $h_1^*,\dots,h_{n-1}^* , e^*_{i,j}$ with $i\neq j$ be the corresponding dual basis of $\fraksl_n^*$.

\begin{proposition}\label{prop: Q T2sln}
For every $n$, 
\[
n \leq \rmQ(T^{(2)}_{\mathfrak{sl}_n}).
\]
\end{proposition}
\begin{proof}
We prove two base cases $\rmQ(T^{(2)}_{\mathfrak{sl}_2})\geq2$ and $\rmQ(T^{(2)}_{\mathfrak{sl}_3})\geq3$. For the case $n=2$, we have
\begin{align*}
T^{(2)}_{\mathfrak{sl}_2}&=2e_{1,2}\otimes h_1^*\otimes e_{1,2}^*+2e_{2,1}\otimes e_{2,1}^*\otimes h_1^*+h_1\otimes e_{1,2}^*\otimes e_{2,1}^* \\
            &-2e_{1,2}\otimes e_{1,2}^*\otimes h_1^* -2e_{2,1}\otimes h_1^*\otimes e_{2,1}^*-h_1\otimes e_{2,1}^*\otimes e_{1,2}^*.
\end{align*}
Consider the restriction given by the following maps:
    \[
    \begin{array}{rlll}
    X_0 : &\mathfrak{sl}_2 & \to &\bbC^2 \\
        &h_1 & \mapsto & 0 \\
        &e_{1,2} & \mapsto & e_0/2 \\ 
        &e_{2,1} &\mapsto & e_1/2 
    \end{array},
    \qquad
        \begin{array}{rlll}
    X_1 : &\mathfrak{sl}_2^* & \to &\bbC^2 \\
        &h_1^* & \mapsto & e_0 \\
        &e_{1,2}^* & \mapsto & 0 \\ 
        &e_{2,1}^* &\mapsto & e_1 
    \end{array},
    \qquad
    \begin{array}{rlll}
    X_2 : &\mathfrak{sl}_2^* & \to &\bbC^2 \\
        &h_1^* & \mapsto & e_1 \\
        &e_{1,2}^* & \mapsto & e_0 \\ 
        &e_{2,1}^* &\mapsto & 0
    \end{array}.
    \]
    One can verify that $(X_0 \otimes X_1 \otimes X_2)(T^{(2)}_{\mathfrak{sl}_2}) = \bfu_3(2)$.

    For the case $n=3$ consider the restriction given by the following maps:
    \[
    \begin{array}{rlll}
    X_0 : &\mathfrak{sl}_3 & \to &\bbC^3 \\
        &h_2 & \mapsto & e_2 \\
        &e_{1,2} & \mapsto & e_0/2 \\ 
        &e_{2,1} & \mapsto & e_1/2 \\
    \end{array},
    \qquad
        \begin{array}{rlll}
    X_1 : &\mathfrak{sl}_3^* & \to &\bbC^3 \\
        &h_1^* & \mapsto & e_0 \\
        &e_{2,3}^* & \mapsto & e_2 \\
        &e_{2,1}^* & \mapsto & e_1 \\
    \end{array},
    \qquad
    \begin{array}{rlll}
    X_2 : &\mathfrak{sl}_3^* & \to &\bbC^3 \\
        &h_1^* & \mapsto & e_1 \\
        &e_{1,2}^* & \mapsto & e_0 \\
        &e_{3,2}^* & \mapsto & e_2
    \end{array}
    \]
    and all the unspecified vectors are mapped to $0$. One can immediately verify that $(X_0 \otimes X_1 \otimes X_2)(T^{(2)}_{\mathfrak{sl}_3}) = \bfu_3(3)$.
  
    If $n \geq 4$, write $n = 2p + 3q$ for some $p,q \geq 0$. Then $T^{(2)}_{\mathfrak{sl}_n}$ restricts to ${T^{(2)}_{\mathfrak{sl}_2}}^{\oplus p} \oplus {T^{(2)}_{\mathfrak{sl}_3}}^{\oplus q}$: in coordinates $\fraksl_2^{\oplus p} \oplus \fraksl_3^{\oplus q}$ embeds in $\fraksl_n$ as block diagonal matrices, with diagonal blocks of size $2$ or $3$. By the first part of the argument, the subrank of ${T^{(2)}_{\mathfrak{sl}_2}}^{\oplus p} \oplus {T^{(2)}_{\mathfrak{sl}_3}}^{\oplus q}$ is bounded from below by $2p + 3q = n$.
\end{proof}

\begin{proposition}\label{prop: uQ Tksl2}
    For $n=2$, $\rmQ(T^{(2)}_{\mathfrak{sl}_2})=2$ and $\uQ(T^{(k)}_{\mathfrak{sl}_2})=2$ for all $k$.
\end{proposition}
\begin{proof}
The case $k = 2$ follows from \cref{prop: Q T2sln}. Let $k\geq3$. Write $h=h_1$, $a=e_{1,2}+e_{2,1}$, and $b=e_{1,2}-e_{2,1}$. Then $h,a,b$ is a basis of $\fraksl_2$; let $h^*,a^*,b^*$ be the dual basis of $\fraksl_2^*$. The result of the bracket on these basis elements is 
\begin{equation}\label{eqn: sl2 bracket hab}
[a,h]=-2b, \qquad [b,a]=2h, \qquad [b,h]=-2a.
\end{equation}
We are going to prove that $T^{(k)}_{\mathfrak{sl}_2}$ degenerates to 
\begin{align*}
a \otimes a^*\otimes (h^*)^{\otimes (k-1)}+b\otimes b^*\otimes(a^*)^{\otimes(k-1)},&\text{ if }k\text{ is odd, and } \\
-b\otimes a^*\otimes (h^*)^{\otimes (k-1)}+h \otimes b^*\otimes(a^*)^{\otimes(k-1)},&\text{ if }k\text{ is even;}
\end{align*}
these are both isomorphic to $\bfu_{k+1}(2)$.

Consider the restriction given by the following maps:
    \[
    \begin{array}{rlll}
    X_0 : &\mathfrak{sl}_2 & \to &\mathfrak{sl}_2 \\
        &h & \mapsto & \frac{1}{2^{k-1}}h \\
        &a & \mapsto & \frac{1}{2^{k-1}}a \\ 
        &b & \mapsto & \frac{1}{2^{k-1}}b 
    \end{array},
    \qquad
        \begin{array}{rlll}
    X_1 : &\mathfrak{sl}_2^* & \to &\mathfrak{sl}_2^* \\
        &h^* & \mapsto & 0 \\
        &a^* & \mapsto & a^* \\ 
        &b^* &\mapsto & b^*
    \end{array},
   \qquad
    \begin{array}{rlll}
    X_{\ell} : &\mathfrak{sl}_2^* & \to &\mathfrak{sl}_2^* \\
        &h^* & \mapsto & h^* \\
        &a^* & \mapsto & a^* \\ 
        &b^* &\mapsto & 0
    \end{array}
    \]
for $\ell=2,\dots,k$.
Let $Y_0 = h^* \otimes h^*$ and $Y_1=a^*\otimes a^*$ be elements of $\fraksl_2^* \otimes \fraksl_2^*$. If $\bfj \in \{0,1\}^q$ is a binary string, write $Y_\bfj = Y_{j_1} \otimes \cdots \otimes Y_{j_q}$, which is an element of $(\fraksl_2^*)^{\otimes 2q}$.

Let $q = \lfloor (k-1)/2 \rfloor$. Then we use induction on $k$ to prove
\begin{align*}
    (X_0\otimes X_1\otimes\cdots X_k)(T^{(k)}_{\mathfrak{sl}_2}) &= a\otimes a^*\otimes h^*\otimes\bigl(\textstyle \sum_{\bfj \in \{0,1\}^{q-1}} Y_\bfj \bigr)\otimes h^* \\
    &-h\otimes a^*\otimes h^*\otimes\bigl(\textstyle \sum_{\bfj \in \{0,1\}^{q-1}} Y_\bfj \bigr)\otimes a^* \\
    &+b\otimes b^*\otimes\bigl(\textstyle \sum_{\bfj \in \{0,1\}^{q}} Y_\bfj \bigr) \quad \text{ if $k$ is odd,} \\
   \text{and } (X_0\otimes X_1\otimes\cdots X_k)(T^{(k)}_{\mathfrak{sl}_2}) &= -b\otimes a^*\otimes h^*\otimes\bigl(\textstyle \sum_{\bfj \in \{0,1\}^{q}} Y_\bfj \bigr) \\
    &+h\otimes b^*\otimes\bigl(\textstyle \sum_{\bfj \in \{0,1\}^{q}} Y_\bfj \bigr)\otimes a^* \\
    &-a\otimes b^*\otimes\bigl(\textstyle \sum_{\bfj \in \{0,1\}^{q}} Y_\bfj \bigr)\otimes h^* \quad \text{ if $k$ is even}.
\end{align*}
To prove the base of the induction for $k=2$, observe that by \eqref{eqn: sl2 bracket hab}, we have 
\begin{align*}
    T^{(2)}_{\mathfrak{sl}_2}=&-2b\otimes a^*\otimes h^* + 2h\otimes b^*\otimes a^* - 2a\otimes b^*\otimes h^* \\
    &+2b\otimes h^*\otimes a^* - 2h\otimes a^*\otimes b^* + 2a\otimes h^*\otimes b^*,
\end{align*}
so that 
\[
(X_0\otimes X_1\otimes X_2)(T^{(2)}_{\mathfrak{sl}_2}) =  -b\otimes a^*\otimes h^* + h\otimes b^*\otimes a^* - a\otimes b^*\otimes h^*.
\]
Now let $k\geq 3$. If $k$ is odd, then $q=\lfloor (k-1)/2 \rfloor = (k-1)/2$; the induction hypothesis applies to $k-1$, which is even, so that $\lfloor (k-2)/2 \rfloor = q-1$. We have 
\begin{align*}
    (X_0\otimes X_1\ootimes X_{k-1})(T^{(k-1)}_{\mathfrak{sl}_2}) = &-b\otimes a^*\otimes h^*\otimes\bigl(\textstyle \sum_{\bfj \in \{0,1\}^{q-1}} Y_\bfj \bigr) \\
    &+h\otimes b^*\otimes\bigl(\textstyle \sum_{\bfj \in \{0,1\}^{q-1}} Y_\bfj \bigr)\otimes a^* \\
    &-a\otimes b^*\otimes\bigl(\textstyle \sum_{\bfj \in \{0,1\}^{q-1}} Y_\bfj \bigr)\otimes h^* .
\end{align*}
Then, 
\begin{align*}
    (X_0\otimes X_1\ootimes X_k)(T^{(k)}_{\mathfrak{sl}_2})&= \textfrac{1}{2}[-b,h]\otimes a^*\otimes h^*\otimes\bigl(\textstyle \sum_{\bfj \in \{0,1\}^{q-1}} Y_\bfj \bigr)\otimes h^*  \\
    &+ \textfrac{1}{2}[-b,a]\otimes a^*\otimes h^*\otimes\bigl(\textstyle \sum_{\bfj \in \{0,1\}^{q-1}} Y_\bfj \bigr)\otimes a^* \\
    &+\textfrac{1}{2} [h,a]\otimes b^*\otimes\bigl(\textstyle \sum_{\bfj \in \{0,1\}^{q-1}} Y_\bfj \bigr)\otimes a^* \otimes a^*  \\
    &+ \textfrac{1}{2}[-a,h]\otimes b^*\otimes\bigl(\textstyle \sum_{\bfj \in \{0,1\}^{q-1}} Y_\bfj \bigr)\otimes h^*\otimes h^*  \\
    &= a\otimes a^*\otimes h^*\otimes\bigl(\textstyle \sum_{\bfj \in \{0,1\}^{q-1}} Y_\bfj \bigr)\otimes h^*  \\
    &- h\otimes a^*\otimes h^*\otimes\bigl(\textstyle \sum_{\bfj \in \{0,1\}^{q-1}} Y_\bfj \bigr)\otimes a^*  \\
    &+ b\otimes b^*\otimes\bigl(\textstyle \sum_{\bfj \in \{0,1\}^{q}} Y_\bfj \bigr). 
\end{align*}
If $k$ is even, then $q = \lfloor(k-1)/2\rfloor = (k-2)/2$, and the induction hypothesis applies to $k-1$, which is odd, and $\lfloor (k-2)/2 \rfloor = q$. The proof is obtained with a calculation similar to the one above.

It remains to show that the restrictions $(X_0 \ootimes X_k)(T^{(k)}_{\mathfrak{sl}_2})$ degenerate to the desired unit tensors. Consider the degeneration defined by the curves
\[
\begin{array}{rlrlcrlrlc}
    g_0: & h & \mapsto & 0 & \text{ if $k$ is odd; } & g_0: & h & \mapsto & h & \text{ if $k$ is even;}\\
     & a & \mapsto & a & \qquad & & a & \mapsto & 0\\
     & b & \mapsto & b & \qquad & & b & \mapsto & b \\
     ~\\
    g_1(\eps): & a^* & \mapsto & a^* & \text{ if $k$ is odd; } & g_1(\eps): & a^* & \mapsto & a^* & \text{ if $k$ is even;}\\
     & b^* & \mapsto & \eps^{q} b^* & \qquad & & b^* & \mapsto & \eps^{-q-2} b^* \\
     ~\\
\end{array}
\]
and, independently of the parity of $k$, for $p = 2 \vvirg k$,
\[
\begin{array}{rlrlc}
 g_p(\eps): & h^* & \mapsto & h^*  \\
 & a^* & \mapsto & \eps^{(-1)^p p}a^*
\end{array}.
\]
Then, if $k$ is odd, so that $q=(k-1)/2$, we obtain
\begin{align*}
(g_0 \otimes g_1(\eps) \ootimes g_k(\eps))& ( X_0 \ootimes X_k)(T^{(k)}_{\mathfrak{sl}_2}) = \\
& a \otimes a^* \otimes h^* \otimes \bigl(\textstyle \sum_{\bfj \in \{0,1\}^{q-1}} \eps^{j_1+\cdots+j_{q-1}} Y_\bfj \bigr) \otimes h^* + \\
\eps^{q} \cdot &b\otimes b^* \otimes \bigl(\textstyle \sum_{\bfj \in \{0,1\}^{q}}  \eps^{-j_1-\cdots-j_q} Y_\bfj \bigr).
\end{align*}
Passing to the limit, we obtain the desired unit tensor, since the terms not vanishing in the limit are the one with $\bfj = (0,\dots,0)$ in the first summation and the one with $\bfj = (1,\dots,1)$ in the second summation.

Similarly, if $k$ is even, so that $q=(k-2)/2$, we obtain
\begin{align*}
(g_0 \otimes g_1(\eps) \ootimes g_k(\eps))& ( X_0 \ootimes X_k)(T^{(k)}_{\mathfrak{sl}_2}) = \\  -&b \otimes a^* \otimes h^* \otimes \bigl(\textstyle \sum_{\bfj \in \{0,1\}^{q}} \eps^{j_1+\cdots+j_{q}} Y_\bfj \bigr)+ \\
(\eps^{-q-2} \cdot &h)\otimes b^* \otimes \bigl(\textstyle \sum_{\bfj \in \{0,1\}^{q}}  \eps^{-j_1-\cdots-j_q} Y_\bfj \bigr) \otimes (\eps^k \cdot a^*).
\end{align*}
As before, passing to the limit, we obtain the desired unit tensor.
\end{proof}

\appendix

\section{Invariant theory of small spaces} \label{sec: invariants}
In this section, we record some known results regarding the invariant theory of $\bbC^3 \otimes \bbC^3 \otimes \bbC^3$ and $\bbC^2 \otimes \bbC^2 \otimes \bbC^2 \otimes \bbC^2$ for the action of $\GL_3^{\times 3}$ and $\GL_2^{\times 4}$ respectively. We refer to \cite[Sec.~3.8.3]{WallachGIT} for the theory.

Orbits for these actions in these spaces have a classification, depending on a small number of parameters. The existence of normal forms follows from the classical theory \cite{Vinberg}, and the explicit normal forms are given in \cite{Nur00a,Nur00b,DitDeGrMar23} for the case of $(\bbC^3)^{\otimes 3}$ and in \cite{VerDehMooVer,CD07,HLTatlas} for the case of $(\bbC^2)^{\otimes 4}$. In the following, we will not need the classifications, but only some results on the rings of invariant polynomials.

\subsection{\texorpdfstring{The case $\bbC^3 \otimes \bbC^3 \otimes \bbC^3$}{}}\label{sec: invariants 333}
Let 
    \[
   S[(\bbC^3)^{\otimes 3}] = \{ F \in \bbC[\bbC^3 \otimes \bbC^3 \otimes \bbC^3] : F = F \circ (g_1 \otimes g_2 \otimes g_3) \text{ for all } g_i \in \SL_3\}
    \]
be the ring of $(\SL_3 \times\SL_3 \times\SL_3)$-invariants on $\bbC^3 \otimes \bbC^3 \otimes \bbC^3$. It is known that $S[(\bbC^3)^{\otimes 3}]$ is freely generated by three elements:
\[
S[(\bbC^3)^{\otimes 3}] = \bbC[F_6,F_9,F_{12}]
\]
where $F_6,F_9,F_{12}$ are three polynomials of degree $6,9,12$ respectively. We provide a direct description, which uniquely determines them. This description depends on a cubic polynomial that can be naturally associated to a tensor $T \in \bbC^3 \otimes \bbC^3 \otimes \bbC^3$, see  \cite{Ng95}: considering the image of the first flattening $T_1 : {\bbC^3}^* \to \bbC^3 \otimes \bbC^3$, the tensor $T$ can be regarded as a $3 \times 3$ matrix of linear forms in $3$ variables; the determinant of such matrix is a cubic polynomial $\phi_T \in S^3 \bbC^3$. The three invariants can be described as follows: 
\begin{itemize}
\item $F_6$ is uniquely determined up to scaling and it is known that $F_6(\bfu_3(3)) \neq 0$; we normalize it so that $F_6(\bfu_3(3)) = 1$.
\item $F_9$ is the equation of the hypersurface of tensors $T$ such that $\uR(T) \leq 4$. It can be realized as the determinant of \emph{Strassen's flattening}, that is the matrix given in \cite[eq. 8]{CGJ}.
\item $F_{12}$ is only defined up to an additive shift by $F_6^2$. It can be chosen to be the equation of the hypersurface of tensors $T$ with the property that $\phi_T$ is a cubic of border (Waring) rank at most $3$.  Since $F_{12}(\bfu_3(3)) \neq 0$, we normalize $F_{12}$ so that $F_{12}(\bfu_3(3)) = 1$.
\end{itemize}

\subsection{\texorpdfstring{The case $\bbC^2 \otimes \bbC^2 \otimes \bbC^2 \otimes \bbC^2$}{}}\label{sec: invariants 2222}
 Let
    \[
   S[(\bbC^2)^{\otimes 4}] = \{ F \in \bbC[\bbC^2 \otimes \bbC^2 \otimes \bbC^2\otimes \bbC^2] : F = F \circ (g_1 \otimes g_2 \otimes g_3 \otimes g_4) \text{ for all } g_i \in \SL_2\}
    \]
be the ring of $(\SL_2 \times\SL_2 \times\SL_2\times\SL_2)$-invariants on $\bbC^2 \otimes \bbC^2 \otimes \bbC^2\otimes \bbC^2$. It is known that $S[(\bbC^2)^{\otimes 4}]$ is freely generated by four elements:
\[
S[(\bbC^2)^{\otimes 4}] = \bbC[F_2,F_4,F'_4, F_{6}]
\]
where $F_2,F_4,F_4',F_6$ are four polynomials of degree $2,4,4,6$ respectively. As in the previous section, we provide a direct description which makes them uniquely determined. For convenience, write $V_1,V_2,V_3,V_4$ for the four copies of $\bbC^2$.
\begin{itemize}
\item $F_2$ is uniquely determined up to scaling and it is known that $F_2(\bfu_4(2)) \neq 0$; we normalize it so that $F_2(\bfu_4(2)) = 1$.\item $F_4$ and $F_4'$ are only defined up to an additive shift by $F_2^2$. They can be chosen to be the determinants of two $4\times 4$ flattenings defined by $T \in V_1 \ootimes V_4$; for instance 
\[
F_4(T) = \det (T_{\{1,2\}} : V_1^* \otimes V_2^* \to V_3 \otimes V_4), \qquad 
F_4'(T) = \det(T_{\{1,3\}} : V_1^* \otimes V_3^* \to V_2 \otimes V_4).
\]
The determinant of the third map $T_{\{1,4\}}$ is a linear combination of $F_4$ and $F_4'$ \cite{Segre}.
\item $F_6$ is only defined up to an additive shift by $F_2^3,F_2F_4,F_2F_4'$. An explicit $F_6$ was constructed in \cite{BerDLaGes} as follows: given $T \in V_1 \ootimes V_4$, regard the image of the bilinear map $T: V_1^* \times V_2^* \to V_3 \otimes V_4$ as a $2 \times 2$ matrix of bilinear forms. The determinant of such matrix is a biquadratic form, which can be regarded as a bilinear form on $S^2 V_1^* \times S^2 V_2^*$, represented by a $3 \times 3$ matrix. The value of $F_6(T)$ is the determinant of such matrix. Different choices and orderings of the spaces yield degree $6$ polynomials lying in the linear span of $F_2F_4,F_2F_4'$ and $F_6$.
\end{itemize}

{\small
\bibliographystyle{alphaurl}
\bibliography{subrank}
}

\end{document}